\documentclass[3p,times]{article}
\usepackage{geometry}
\usepackage{authblk}  
\usepackage{physics}
\usepackage{amsmath}
\usepackage{amsthm}
\usepackage{graphicx}
\usepackage{caption}
\usepackage{amsfonts}
\usepackage{bm}
\usepackage{mathbbol}
\usepackage{dsfont}
\usepackage{subcaption}
\usepackage[dvipsnames]{xcolor}
\usepackage[section]{placeins}

\usepackage{hyperref} 
\hypersetup{
	 colorlinks=true,             
	 linkcolor=blue,
	 citecolor=NavyBlue,
	 urlcolor=blue }

\newtheorem{theorem}{Theorem}
\newtheorem{remark}{Remark}

\newcommand{\vv}{\bm{v}}
\newcommand{\xx}{\bm{x}}

\newcommand{\oxx}{\bm{\overline{x}}}
\newcommand{\comment}[1]{}

\newcommand{\old}[1]{{#1}^{\mathrm{old}}}
\newcommand{\sgrad}[2]{ \langle \nabla #1 \rangle_{#2}}
\newcommand{\wgrad}[2]{ \langle \widetilde{\nabla} #1 \rangle_{#2}}
\newcommand{\wdiv}[2]{ \langle \widetilde{\nabla} \cdot #1 \rangle_{#2}}
\newcommand{\lap}[2]{ \langle \Delta #1 \rangle_{#2}}

\newcommand{\A}{\mathbf{A}}
\renewcommand{\b}{\mathbf{b}}
\newcommand{\B}{\mathbf{B}}

\DeclareMathOperator{\dist}{dist}
\DeclareMathOperator{\diam}{diam}

		\title{\bf Semi-implicit Lagrangian Voronoi Approximation for the incompressible Navier-Stokes equations}
		
		\author[1]{Ondřej Kincl} 
		\author[1]{Ilya Peshkov}
		\author[2,3]{Walter Boscheri}
		\affil[1]{\small Laboratory of Applied Mathematics, DICAM, University of Trento, via Mesiano 77, 38123 Trento, Italy}
		\affil[2]{Laboratoire de Mathématiques UMR 5127 CNRS, Université Savoie Mont Blanc, 73376 Le Bourget du Lac, France}
            \affil[3]{Department of Mathematics and Computer Science, University of Ferrara, Ferrara, Italy}
		\date{\today}
		
\begin{document}
		\maketitle
		
		\begin{abstract}
            We introduce Semi-Implicit Lagrangian Voronoi Approximation (SILVA), a novel numerical method for the solution of the incompressible Euler and Navier-Stokes equations, which combines the efficiency of semi-implicit time marching schemes with the robustness of time-dependent Voronoi tessellations. In SILVA, the numerical solution is stored at particles, which move with the fluid velocity and also play the role of the generators of the computational mesh. The Voronoi mesh is rapidly regenerated at each time step, allowing large deformations with topology changes. As opposed to the reconnection-based Arbitrary-Lagrangian-Eulerian schemes, we need no remapping stage. A semi-implicit scheme is devised in the context of moving Voronoi meshes to project the velocity field onto a divergence-free manifold. We validate SILVA by illustrative benchmarks, including viscous, inviscid, and multi-phase flows. Compared to its closest competitor, the Incompressible Smoothed Particle Hydrodynamics (ISPH) method, SILVA offers a sparser stiffness matrix and facilitates the implementation of no-slip and free-slip boundary conditions.
		\end{abstract}
		%
	
	\section{Introduction}
	In computational fluid dynamics, Lagrangian methods are invaluable for their ability to handle complex deformations and material interfaces. The defining feature of Lagrangian methods is that the degrees of freedom of the spatial approximation follow the motion of the material, rather than being static with respect to a selected Eulerian frame of reference. Consequently, the (often problematic) nonlinear convective terms are absent, thus obtaining a family of numerical schemes which exhibit virtually no numerical dissipation at contact discontinuities and material interfaces. Moreover, individual particles or cells preserve their identity, facilitating the mass conservation and the implementation of sharp material interfaces.
	
	Developing a mesh-based Lagrangian method is notoriously difficult because of the large deformations inherent to the fluid motion. This problem becomes particularly evident when considering flows with high vorticity or large velocity gradients, that stretch, twist or compress the control volumes until they get invalid, i.e. with negative volumes. To overcome this problem, mesh optimization algorithms have been developed \cite{winslow1963equipotential,winslow1966numerical,dobrev2019target,galera2010two}, aiming at improving the mesh quality by performing a sort of regularization of the shape of each single cell. If the moving mesh scheme belongs to the category of Arbitrary-Lagrangian-Eulerian (ALE) methods, in which the mesh motion can be driven independently from the fluid velocity, the stage of mesh optimization is referred to as \textit{rezoning}. The rezoning step is directly embedded in the numerical scheme for direct ALE methods \cite{boscheri_LagrangeFV,boscheri_LagrangeDG}, otherwise it must be followed by a \textit{remapping} step in indirect ALE methods \cite{dukowicz1987accurate,loubere2005subcell,kucharik2014conservative}.
	These techniques are all based on the assumption that the optimization process does not change the topology of the underlying computational grid, hence posing a severe limit to the possibility of improving the mesh quality for complex configurations. Moreover, the mesh optimization stage has to be compatible with some physical requirements on the numerical solution, e.g. contact discontinuities must remain untouched to ensure the sharp interface resolution of Lagrangian schemes. The most general solution to this difficulty, which allows to combine mesh optimization and Lagrangian mesh motion, is to allow for \textit{topology changes}. This approach is also known as reconnection, that has been originally presented in \cite{pasta1959heuristic}. The idea was to use a set of Lagrangian particles, that are used to define the computational tessellation without assuming a fix connectivity: indeed, the topology of the mesh was time dependent, so that at each time step the connectivity was established by defining the surrounding domains (or control volumes) in accordance to the current particle positions. This mesh was then used to discretize the equations written in Lagrangian form. More recent developments of this seminal idea are given by Free-Lagrange methods \cite{fritts1985free,crowley2005flag}, Re-ALE schemes \cite{loubere2010reale,bo2015adaptive} and finite volume ALE methods \cite{springel2011hydrodynamic,gaburro2020high}.

	Among many others, a well-established mesh-free Lagrangian method is the Smoothed Particle Hydrodynamics (SPH), developed in 1970s: see \cite{monaghan1992smoothed} for a review. In the SPH framework, the continuum is replaced by a finite number of material points (particles), and the gradients of physical variables (like density and velocity) are approximated using convolution and a convenient set of \textit{kernel functions}. No computational grid is used in the SPH framework, hence yielding high versatility in the simulation of large deformations of the continuum. Due to its particle nature, the SPH method avoids the mesh-related issues and it is exceptionally robust. Owing to this advantage, together with the simplicity of implementation, SPH boasts with wide range of applications in science and industry \cite{monaghan2012smoothed,bavsic2022sphINS}. Nevertheless, compared to mesh-based approaches, the computational cost of SPH schemes associated to each degree of freedom is typically larger, since every particle needs to communicate with every neighbor within a certain radius (as specified by the smoothing length) \cite{violeau2016smoothed}. Moreover, converging to the exact solution is challenging because the spatial step must shrunk faster than the smoothing length \cite{zhu2015numerical}, leading to an unfavorable effective convergence rate. To remain competitive in this regard, the accuracy of SPH methods needs to be improved by relying on a set of re-normalized operators \cite{kulasegaram2000corrected, oger2007improved}, typically at the expense of discrete energy conservation \cite{violeau2012fluid}. Lastly, let us mention the problematic of boundary conditions. While free surfaces in fluids and solids are automatically handled, implementing Dirichlet or Neumann conditions in the SPH framework is less trivial (in comparison to e.g. Finite Element Method). As it turns out, having a mesh is not completely without merit and this is where Lagrangian Voronoi Methods (LVM) enter the game.
	
	The study of Lagrangian methods based on Voronoi tessellations dates back to the turn of the millennium \cite{ball1996free, Howell2002} and received further attention in the light of recent publications 
	\cite{hess2010particle,springel2011hydrodynamic,gaburro2020high,despres2024lagrangian}. By 
	definition, a Voronoi cell $\Omega_i$ is the set of points which is closer to the generating particle 
	(seed) $\xx_i$ than to any other particle $\xx_j$. The Voronoi mesh generated by moving particles 
	has the advantage of keeping reasonably good quality even in case of large deformations. Furthermore, the local topology of Voronoi meshes can be arbitrary, thus this kind of mesh permits to remove some mesh imprinting that might arise from the adoption of the Lagrangian frame of reference.
	
	While the aforementioned studies revolve around compressible fluids and use explicit time steps, in this 
	paper, we introduce some implicitness to obtain a LVM variant for the incompressible Euler and Navier-Stokes equations. With the aim of avoiding the solution of globally implicit systems defined on the entire computational grid, the implicit-explicit (IMEX) approach has emerged in the literature \cite{Hofer,PR_IMEX,Dumbser_Casulli16,BosPar2021,izzo2021strong,izzo2021construction,izzo2017highly} as a powerful time integration technique which permits to split explicit and implicit fluxes \cite{Toro_Vazquez12,Dumbser_Casulli16,BosPar2021}.
	In particular, we rely on a semi-implicit time discretization \cite{Casulli1990,ParkMunz2005,VoronoiDivFree,carlino2024arbitrary}, in which the mesh motion is treated explicitly while the pressure terms are taken into account implicitly. This is well suited to enforce the divergence-free condition on the velocity field, which is nothing but the energy equation for incompressible fluids.
	
	To the best of our knowledge, such studies are missing in the literature. Our novel numerical scheme is linked with the recent work of Després \cite{despres2024lagrangian}, and it relies 
	on the improved gradient approximation by Serrano and Espanol \cite{serrano2001thermodynamically}. The scheme presented here is termed Semi-Implicit Lagrangian Voronoi Approximation (SILVA), since the mathematical model is given by the incompressible Navier-Stokes equations. We will show that SILVA arises quite naturally from basic assumptions. The correctness and accuracy 
	of the scheme is verified in benchmarks, including the lid-driven cavity problem and a Rayleigh-Taylor instability test for multi-phase fluids. Not only engineers and researchers interested in new Lagrangian numerical methods for fluid mechanics may find this paper valuable, but also SPH specialists because of the deep analogies between SILVA and incompressible SPH (ISPH). Both methods involve moving particles, are based on Helmholtz decomposition, and the cell-list structure from ISPH is a perfect tool for constructing a Voronoi grid for SILVA. To summarize, the two methods are highly inter-compatible and an implementation of a physical phenomenon in one method can be painlessly imported to the other one.
	
	The structure of the paper is as follows. In Section \ref{sec.pde} we present the governing equations of incompressible fluids, while in Section \ref{sec.VoronoiMesh} we introduce the definition and notation of the moving Voronoi mesh, including its generation. Section \ref{sec.numscheme} is devoted to the description of the numerical scheme, namely our Lagrangian Voronoi Method, with details on the explicit and the implicit time discretization as well as on the discrete spatial operators. The numerical validation of the accuracy and robustness of SILVA is presented in Section \ref{sec.test} by running several benchmarks for incompressible fluids. Finally, Section \ref{sec.concl} finalizes this article by summarizing the work and giving an outlook to future investigations.

	\section{Governing equations} \label{sec.pde}
	We consider a $d$-dimensional domain $\Omega \subset \mathds{R}^d$ closed by the boundary $\partial \Omega \subset \mathds{R}^{d-1}$, with $\xx$ being the spatial coordinate vector and $t \in \mathds{R}_0^+$ representing the time coordinate. 
	In the Lagrangian frame, the mathematical model which describes the phenomena of incompressible viscous flows writes
	\begin{subequations}\label{eqn.pde}
		\begin{align}
			\frac{\text{d}\vv}{\text{d}t} + \frac{\nabla p}{\rho} - \nu \Delta \vv &= \mathbf{0}, \label{eqn.pde_v} \\
			\nabla \cdot \vv &= 0, \label{eqn.divv} \\
			\frac{\text{d}\xx}{\text{d}t} &=\vv. \label{eqn.pde_x} 
		\end{align}
	\end{subequations}
	Here, $\vv(\xx,t)$ identifies the velocity vector, $p(\xx,t)$ denotes the pressure, and $\rho(\xx, t)$ is the fluid density, $\nu$ is the kinematic viscosity of the fluid and
	\begin{equation}
		\frac{\text{d}g}{\text{d}t} = \frac{\partial g}{\partial t} + \vv \cdot \nabla g.
	\end{equation}
	is the material derivative of a quantity $g(\xx,t)$. Equation \eqref{eqn.pde_x} is the so-called trajectory equation, which drives the motion of the coordinates and is the defining feature of a Lagrangian method.
	
	The flow is incompressible, which implies that $\rho$ is constant along the flow trajectories, although it is not necessarily constant in space. The velocity field is subjected to the divergence constraint \eqref{eqn.divv}, which can be reformulated in a weak sense. Indeed, by multiplying Equation \eqref{eqn.divv} by a smooth test function $\varphi(\xx) \in \mathcal{C}^{\infty}$ and integrating by parts over the domain $\Omega$, we have that 
	\begin{equation}
		\int_{\Omega} \varphi \, \nabla \cdot \vv \, \dd \xx = \int_{\partial \Omega} \phi \, \vv \cdot \bm{n} \, \dd S - \int_{\Omega} \nabla \varphi \cdot \vv \, \dd \xx.
		\label{eqn.divv_weak}
	\end{equation}
	We will assume no fluid penetration through boundary 
	\begin{equation}
		\vv \cdot \bm{n} = 0, \quad \forall \xx \in \partial \Omega,
	\end{equation}
 	so that Equation \eqref{eqn.divv_weak} reduces to
	\begin{equation}
		\int_{\Omega} \nabla \varphi \cdot \vv \, \dd \xx = 0, \qquad \forall \varphi \in \mathcal{C}^{\infty}.
		\label{eqn.pde_divv}
	\end{equation}
	
	\section{Moving Voronoi mesh with topology changes} \label{sec.VoronoiMesh}
	
	\subsection{Notation and definitions}
	Let assume $\Omega \subset \mathds{R}^d$ to be an open convex polytope. We consider $N$ \textit{distinct} sample points $\xx_1, \xx_2,\dots, \xx_N \in \Omega$, which will be treated as material points. The points $\xx_i$ are called \textit{seeds}, for $i=1,2,\dots,N$. A Voronoi cell $\omega_i$ corresponding to the seed $\xx_i$ is defined for every $i,j=1,2,\ldots,N$ as
	\begin{equation*}
		\omega_i = \bigcap_{j \neq i} \bigg\{ \xx \in \Omega: \quad |\xx - \xx_i| < |\xx - \xx_j| \bigg\}.
	\end{equation*} 
	To make notation easier throughout the entire paper, we assume that the cell indexes are always running in the interval $i,j=1,2,\ldots,N$. In the Lagrangian framework, every cell contains a portion of fluid with mass $M_i$, which is constant in time. The Voronoi cell is a convex polytope and its boundary $\partial \omega_i$ is composed of the set of facets
	\begin{equation}
		\Gamma_{ij} = \Gamma_{ji} = \bigg\{ \xx \in \partial \omega_i: \quad |\xx - \xx_i| = |\xx - \xx_j| \bigg\}.
	\end{equation}
	Furthermore, the cell boundary $\partial \omega_i$ might share a portion of the domain boundary $\partial \Omega$ when $ \partial \Omega \cap \partial \omega_i \neq 0$. In such case, we refer to $\omega_i$ as a \textit{boundary-touching cell}, or an \textit{interior cell} otherwise. We also define the surface element:
	\begin{equation}
		\label{eqn.dS}
		\bm{S}_i = \int_{\partial \omega_i} \bm{n} \, \dd S,
	\end{equation}
	where $\bm{n}(\xx)$ is the outward pointing unit normal vector defined on $\partial \omega_i$.
	
	The golden property of Voronoi meshes is the orthogonality between the facets $\Gamma_{ij}$ and the segments joining two seeds, namely $\xx_{ij} = \xx_i - \xx_j$. Hence, the outer normal vector $\bm{n}_{ij}$ is simply given for each facet $\Gamma_{ij}$ by
	\begin{equation}
		\label{eqn.nv}
		\bm{n}_{ij} = -\frac{\xx_{ij}}{r_{ij}},
	\end{equation}
	where $r_{ij} = |\xx_{ij}|$ represents the inter-seed distance. The $d$-dimensional Lebesgue measure of each cell is denoted by $|\omega_i|$, while $|\partial \omega_i|$ is the $(d-1)$-dimensional Hausdorff measure of its boundary. Likewise, $|\Gamma_{ij}|$ represents the length of the facet $\Gamma_{ij}$. 
	Since our Voronoi meshes are generally not central, we must distinguish the midpoint of a facet
	\begin{equation}
		\bm{m}_{ij} = \frac{1}{|\Gamma_{ij}|}\int_{\Gamma_{ij}} \xx \; \dd S, 
	\end{equation}
	and the inter-seed midpoint
	\begin{equation}
		\label{eqn.xij}
		\oxx_{ij} = \frac{\xx_i + \xx_j}{2}.
	\end{equation}
	Indeed, for a general Voronoi cell, it is possible that $\oxx_{ij} \notin \Gamma_{ij}$. 
	
	To evaluate partial derivatives in closed form for a Voronoi mesh, we refer to the recent developments forwarded in \cite{despres2024lagrangian}. There is indeed a simple formula for computing the derivative of a Voronoi volume with respect to the other seeds \cite{despres2024lagrangian}, which is recalled in the following theorem.
	
	\begin{theorem}
		For $j \neq i$ we have the closed form formulae:
		\begin{equation}
			\pdv{|\omega_i|}{\xx_j} = -\frac{|\Gamma_{ij}|}{r_{ij}} \left( \bm{m}_{ij} - \xx_j \right)
			\label{eq:Omega_wrt_j}
		\end{equation}
		and
		\begin{equation}
			\pdv{|\omega_i|}{\xx_i} = -\sum_{j\neq i} \pdv{|\omega_j|}{\xx_i} = -\sum_{j\neq i} \pdv{|\omega_i|}{\xx_j} - \bm{S}_i
			\label{eq:Omega_wrt_i}.
		\end{equation}
	\end{theorem}
	\begin{proof}
		We refer to \cite{despres2024lagrangian} for the proofs of \eqref{eq:Omega_wrt_j} and the identity
		\begin{equation}
			\pdv{|\omega_i|}{\xx_i} = -\sum_{j\neq i} \pdv{|\omega_i|}{\xx_j} - \bm{S}_i.
		\end{equation}
		The remaining identity
		\begin{equation}
			\pdv{|\omega_i|}{\xx_i} = -\sum_{j\neq i} \pdv{|\omega_j|}{\xx_i} 
		\end{equation}
		is a simple consequence of $\sum_i |\omega_i| = |\Omega|$.
	\end{proof}
	We note that formula \eqref{eq:Omega_wrt_i} fails when $\Omega$ is non-convex. This is quite a concerning limitation, which will require special treatment in future studies. In this paper, all domains are rectangular.
	
	\subsection{Voronoi mesh generation}
	The problem of generating Voronoi meshes is well studied in the literature \cite{aurenhammer2013voronoi} and Voronoi mesh generators have been implemented in a variety of open-source libraries, such as Qhull \cite{barber2013qhull}, Voro++ \cite{rycroft2009voro++} or CGAL \cite{fabri2009cgal}. Most of them rely on constructing the Delaunay triangulation, which is a topological dual of a Voronoi grid. However, for a Lagrangian Voronoi method, we 
	found that a direct approach \cite{ray2018meshless, liu2020parallel} is preferable, where no dual mesh is constructed. We illustrate this iterative technique in Figure 
	\ref{fig:voromeshing} for the first five iterations $k=\{1,2,3,4,5\}$. In the first step, i.e. for $k=0$, every cell 
	$\omega_i$ is initialized as $\omega_i^{0} = \Omega$. We then construct a cell list, which divides 
	the computational domain $\Omega$ into a finite number of partitions (squares in Figure \ref{fig:voromeshing}) of size comparable to the spatial 
	resolution $\delta r$ (in this paper, we set the side length of those partitions to be $2\delta 
	r$). For a given $\xx_i$ (green particle in Figure \ref{fig:voromeshing}), we find the containing partition $P_{i,0} \ni \xx$ and proceed by generating 
	an iterator through a sequence $ P_{i,k}$, $k=1,2,3,\ldots$
	of all other partitions in the cell list, sorted in ascending order by their distance to $P_{i,0}$. Then, the polytopes 
	\begin{equation}
		\omega_i^{k} = \bigcap_{\xx_j \in P_{i,k} } \left\{ \xx \in \omega_i^{k-1} : |\xx - \xx_i| < |\xx - \xx_j|  \right\}, \quad k=1,2,3,\ldots 
	\end{equation}
	are constructed in an iterative process. Each $\omega_i^{k}$ has a radius $R_i^{k}$, which is the distance between $\xx_i$ and the furthest vertex of $\omega_i^{k}$. The iterative procedure stops when the Hausdorff distance between $P_{i,0}$ and $P_{i,k+1}$ satisfies
	\begin{equation}
		R_i^{k} < \frac{1}{2} \dist \left( P_{i,0}, P_{i,k+1}\right),
	\end{equation}
	thus we set $\omega_i = \omega_i^{k}$, as all remaining points are provably too remote and can be excluded from the computation. 
	
	\begin{figure}[ht!]
		\centering
		\begin{subfigure}{0.33\linewidth}
			\includegraphics[trim = {1cm, 0, 1cm, 0}, width=\linewidth]{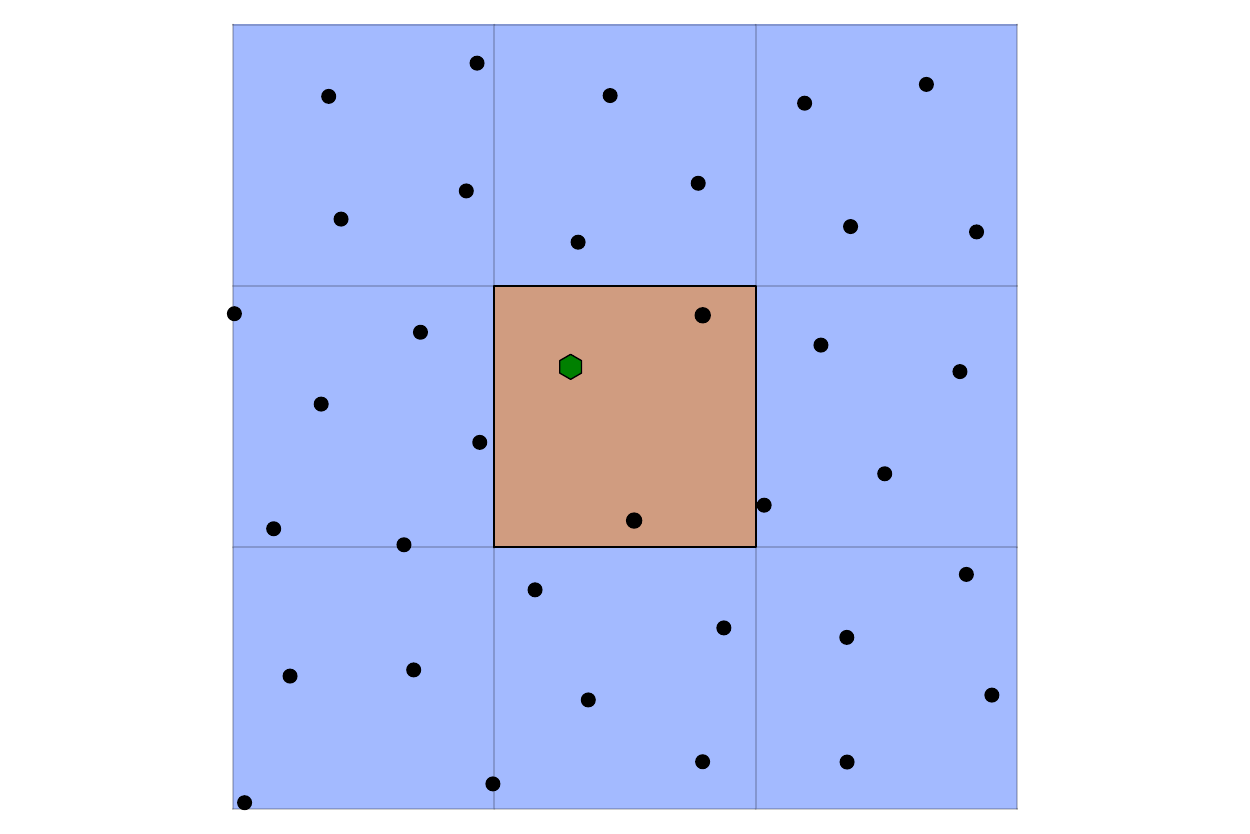}
		\end{subfigure}%
		\begin{subfigure}{0.33\linewidth}
			\includegraphics[trim = {1cm, 0, 1cm, 0}, width=\linewidth]{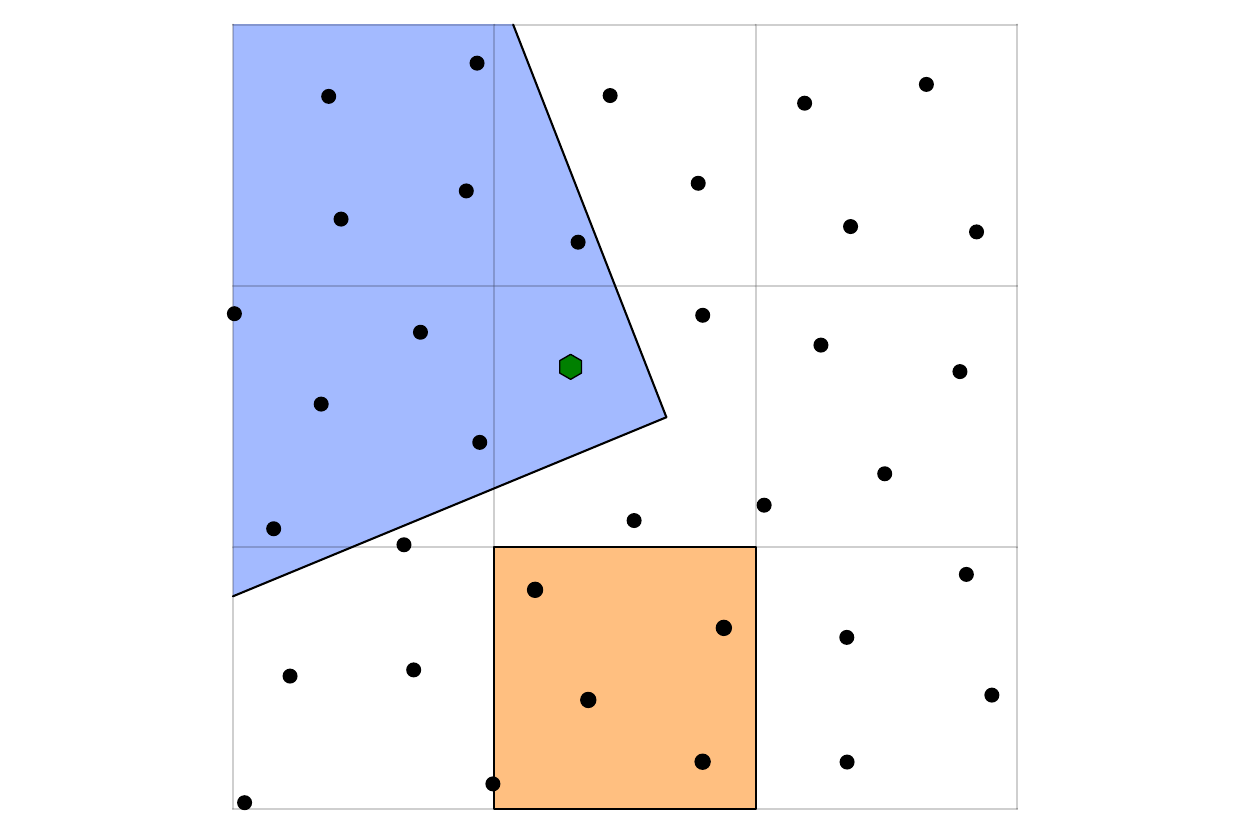}
		\end{subfigure}
		\begin{subfigure}{0.33\linewidth}
			\includegraphics[trim = {1cm, 0, 1cm, 0}, width=\linewidth]{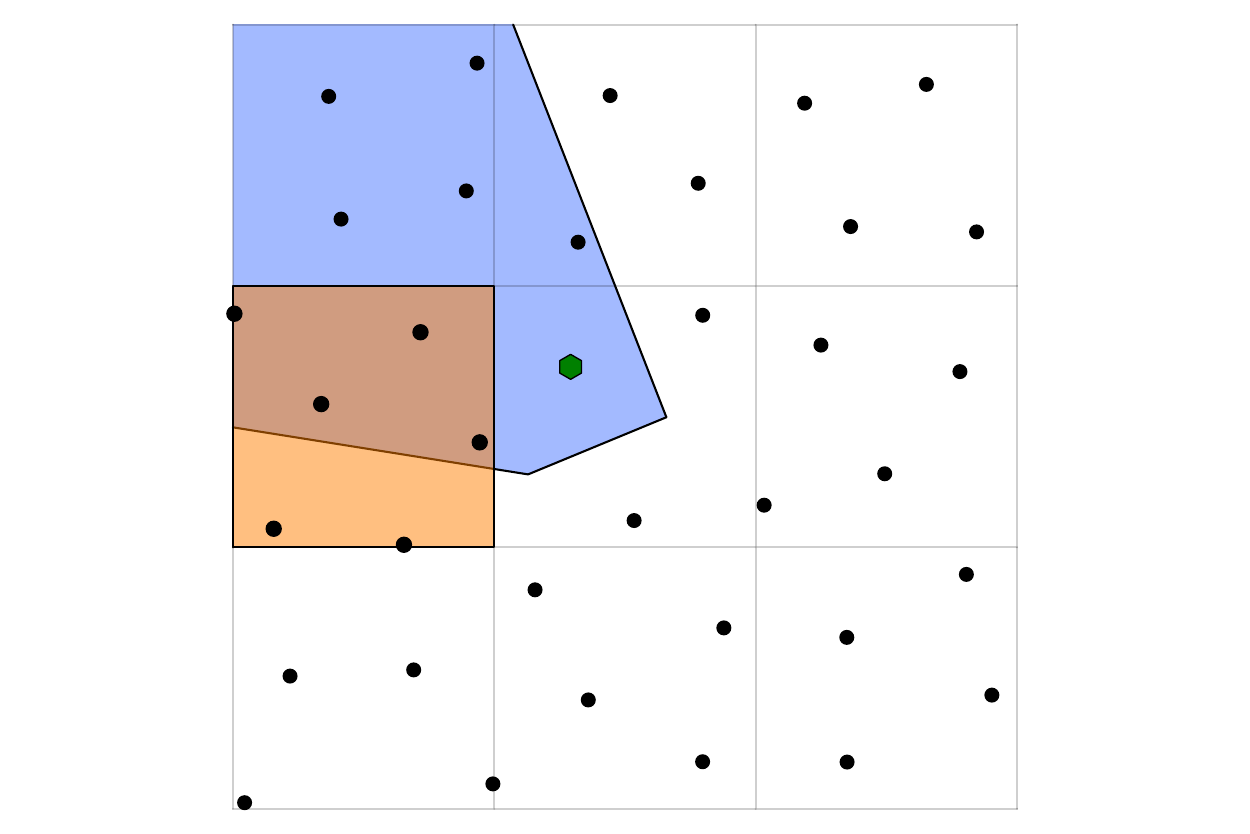}
		\end{subfigure}
		
		\begin{subfigure}{0.33\linewidth}
			\includegraphics[trim = {1cm, 0, 1cm, -2cm}, width=\linewidth]{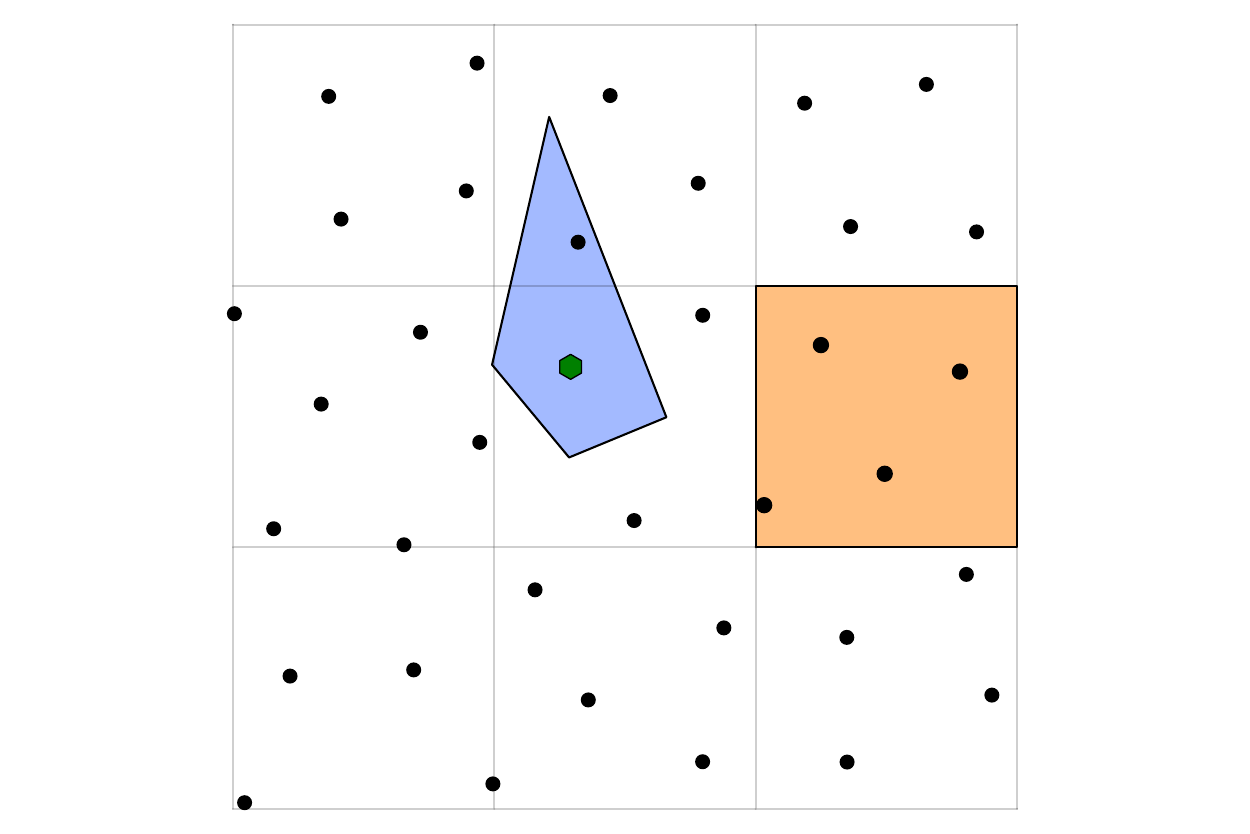}
		\end{subfigure}%
		\begin{subfigure}{0.33\linewidth}
			\includegraphics[trim = {1cm, 0, 1cm, -2cm}, width=\linewidth]{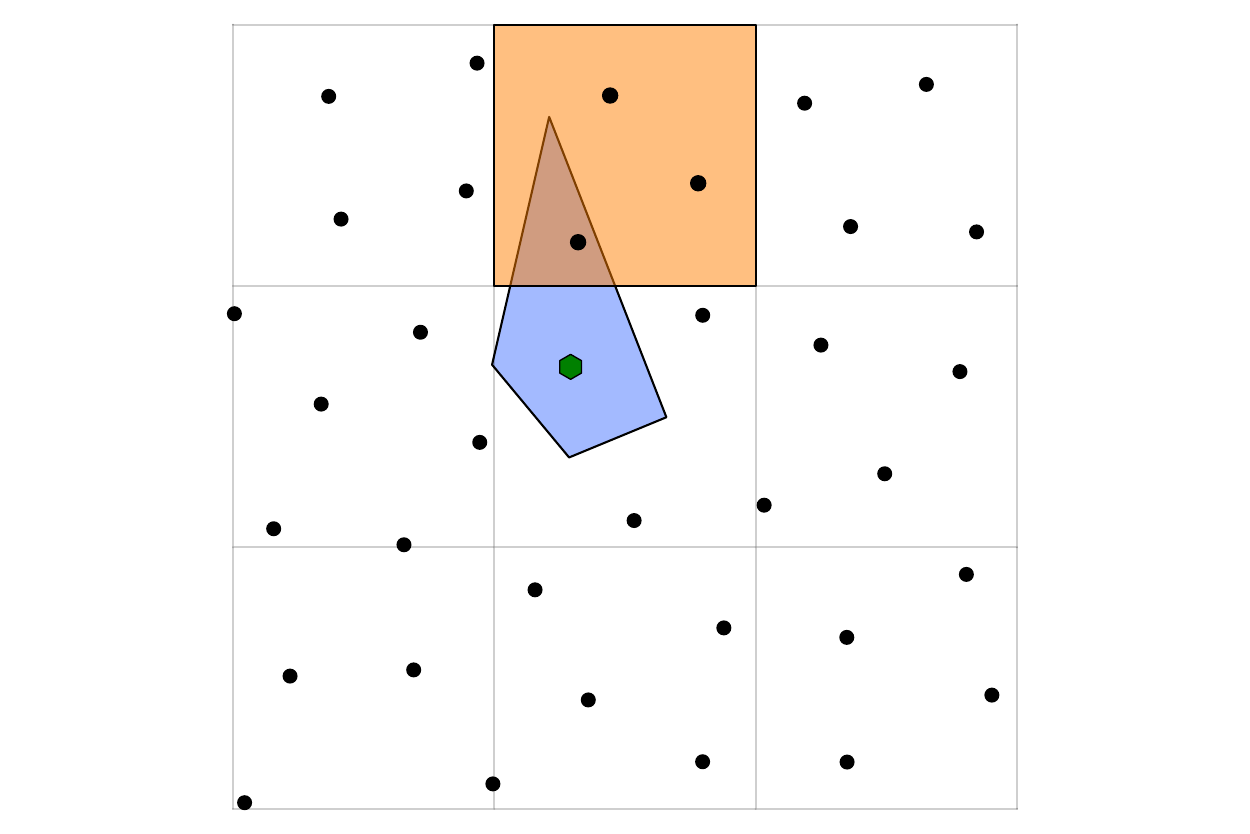}
		\end{subfigure}
		\begin{subfigure}{0.33\linewidth}
			\includegraphics[trim = {1cm, 0, 1cm, -2cm}, width=\linewidth]{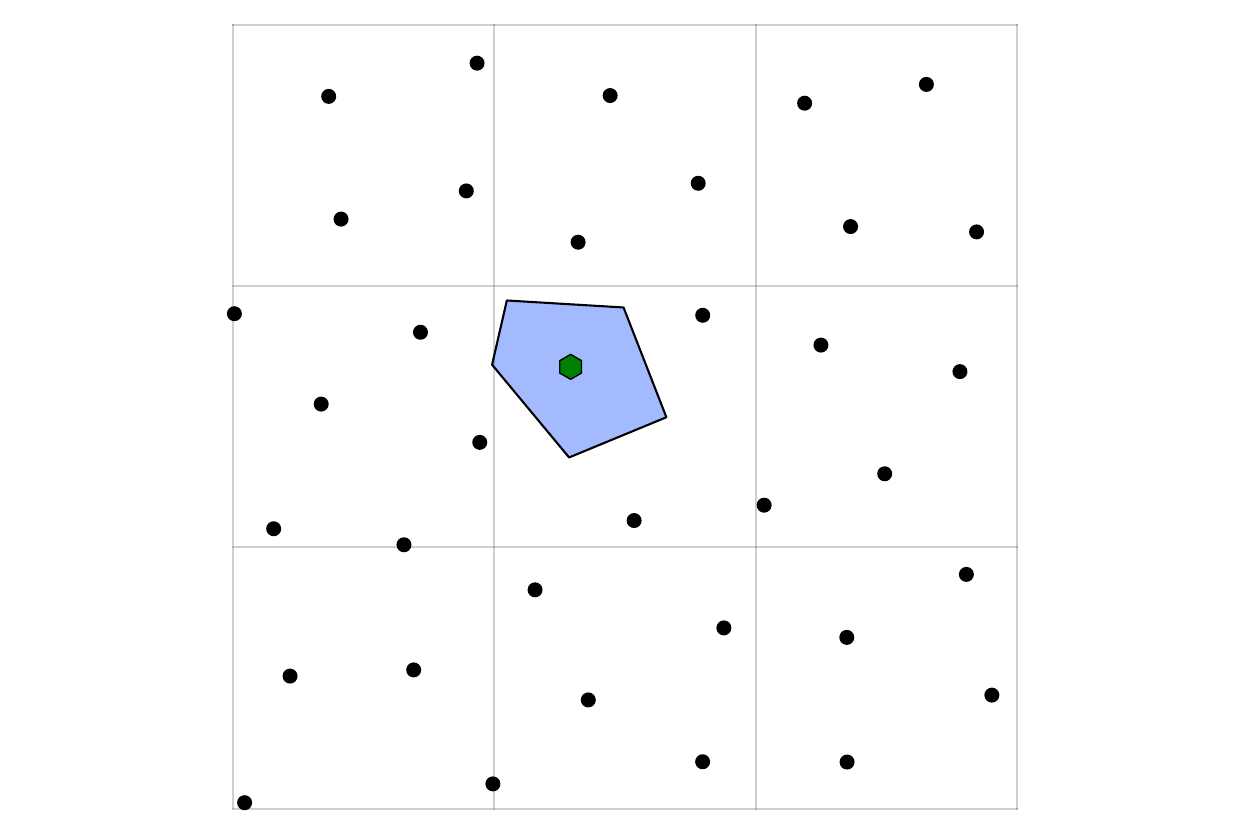}
		\end{subfigure}
		\caption{A Voronoi cell for the green particle in the center can be obtained by intersecting $N-1$ half-planes, one for each other seed. The cell list structure optimizes the computation by focusing on nearby seed only. The orange and blue colors highlight the partition $P_{i,k}$ and the incomplete Voronoi cell $\omega_i^k$ respectively for $k = 0, 1, \dots, 5$.}
		\label{fig:voromeshing}
	\end{figure}
	
	In the Lagrangian approach, the seeds are displaced with the fluid velocity. Therefore, the Voronoi mesh needs to be \textit{regenerated} at each time step with possible changes in the mesh connectivity. Starting from the second time step, the process of mesh generation is further accelerated by initializing 
	\begin{equation}
		\omega^{0}_i = \bigcap_{j \in \old{N}_i } \left\{ \xx \in \Omega : |\xx - \xx_i| < |\xx - \xx_j|  \right\},
	\end{equation}
	where $\old{N}_i$ is the set of neighbors of $\omega_i$ from the previous mesh configuration. With this heuristics, the number of non-trivial half-plane cuts is reduced to bare minimum. In the context of numerical simulations on moving Voronoi tessellations, the method has the following advantages.
	\begin{itemize}
		\item If $\diam(\omega_i/\delta r)$ is bounded, $O(N)$ complexity is guaranteed in the mesh generation process.
		\item Since the cell constructions are independent processes, the method can be easily parallelized. 
		\item This technique works both in 2 and 3 dimensions.
	\end{itemize}

	\section{Lagrangian Voronoi Method} \label{sec.numscheme}
	In this section, we describe and analyze the discrete gradient operator on moving Voronoi grids for a generic smooth function. Then, we present the semi-discrete in space numerical scheme, and finally we provide the details of the semi-implicit time marching algorithm. For the sake of clarity, we first assume no viscous terms in the velocity equation \eqref{eqn.pde_v}, hence focusing on the incompressible Euler model. Viscous forces will be consistently considered in the numerical scheme and presented at the end of this section.

	\subsection{Discrete gradient operator}
	Let us assume that $p(\xx)$ is a smooth function defined on $\Omega$, and $p_i = p(\xx_i)$ are its point values. In order to approximate $\nabla p(\xx_i)$, we use the following representation of the identity matrix:
	\begin{equation}
		\mathbb{I} = \frac{1}{|\omega_i|} \int_{\partial \omega_i}  (\xx - \xx_i) \otimes \bm{n} \, \dd{S},
	\end{equation}
	which holds for every cell $\omega_i$ by virtue of the divergence theorem. For any interior cell $\omega_i$, using the definition \eqref{eqn.nv} we obtain:
	\begin{align}
			\nabla p(\xx_i) &= \frac{1}{|\omega_i|} \int_{\partial \omega_i}  (\xx - \xx_i) \; \left(\nabla p(\xx_i) \cdot \bm{n} \right) \, \dd{S} \nonumber\\
			&= -\frac{1}{|\omega_i|} \sum_j \frac{|\Gamma_{ij}|}{r_{ij}}(\bm{m}_{ij} - \xx_i) \left(\nabla p(\xx_i) \cdot \xx_{ij} \right)\nonumber\\
			& \approx -\frac{1}{|\omega_i|} \sum_j \frac{|\Gamma_{ij}|}{r_{ij}}p_{ij}(\bm{m}_{ij} - \xx_i) =: \sgrad{p}{i},
		\label{eq:sgrad}
	\end{align}
	where $p_{ij} = p_i - p_j$. This is the definition of the discrete gradient operator $\sgrad{p}{i}$. It is clear that this approximation is exact for polynomials of first degree. Additionally, we can provide the following error estimate.
	\begin{theorem}
		Let $p(\xx) \in \mathcal{C}^2$, i.e. $p$ is twice continuously differentiable, and let assume $d=2$ in two space dimensions. Then for any interior cell $\omega_i$ we have that
		\begin{equation}
			|\nabla p(\xx_i) - \sgrad{p}{i}| \leq 2d \, \|\nabla^2 p \|_{\infty} \diam(\omega_i).
			\label{eq:serror}
		\end{equation}
	\end{theorem}
	\begin{proof}
		Using the idea of \eqref{eq:sgrad}, we have
		\begin{equation}
			\begin{split}
				|\nabla p(\xx_i) - \sgrad{p}{i}| &\leq \frac{1}{|\omega_i|} \sum_j |\Gamma_{ij}| \, r_{ij} \, \; \|\nabla^2 p \|_{\infty} \diam (\omega_i),
			\end{split}
		\end{equation}
		and we get the desired result by combining this with a formula for the Voronoi cell volume:
		\begin{equation}
			|\omega|_i = \frac{1}{2d}\sum_j |\Gamma_{ij}| \, r_{ij}.
		\end{equation}
	\end{proof}
	\begin{remark}
		Intriguingly, the error term does not depend on the mesh quality. We do not require to bound the number of facets or the ratios $\diam(\omega_i)/r_{ij}$.
	\end{remark}
	\begin{remark}
		If the Voronoi mesh is rectangular with the sides $\Delta x$ and $\Delta y$, then \eqref{eq:sgrad} becomes equivalent to the central finite difference operator
			\begin{equation}
				\sgrad{p}{i} =\left( \begin{array}{c} 
					\frac{p(x+\Delta x, y) - p(x-\Delta x,y)}{2\Delta x} \\[2.5mm]
					\frac{p(x, y+\Delta y) - p(x,y-\Delta y)}{2\Delta y}
				\end{array} \right),
			\end{equation}
		which has second order of accuracy. As we shall demonstrate in Section \ref{sec:tagr}, it is possible to obtain a super-linear convergence in velocity if a rectangular grid is used in the initialization. This is understandable, because with decreasing spatial step, the mesh eventually becomes "locally rectangular" throughout the entire simulation time. 
	\end{remark}
	
	
	\begin{remark}
		By adding a ``clever zero'', we can obtain an alternative formulation (valid for interior cells):
		\begin{equation}
			\begin{split}
				&\sgrad{p}{i} = \sgrad{p}{i} - \int_{\partial \omega_i} p_i \bm{n} \, \dd S= -\frac{1}{|\omega_i|} \sum_j \frac{|\Gamma_{ij}|}{r_{ij}}\left(p_{ij}(\bm{m}_{ij} - \xx_i) +p_i \xx_{ij} \right) \\
				& =-\frac{1}{|\omega_i|} \sum_j \frac{|\Gamma_{ij}|}{r_{ij}}\left(p_{ij}(\bm{m}_{ij} - \oxx_{ij}) + \overline{p}_{ij} \xx_{ij}\right),
			\end{split}
		\end{equation}
		where $\overline{p}_{ij} = \frac{1}{2}(p_i + p_j)$ and $\oxx_{ij}$ is given by \eqref{eqn.xij}. This form of the gradient approximation is used in the work of Serrano et al. \cite{serrano2001thermodynamically} and Springel \cite{springel2011hydrodynamic}.
	\end{remark}
	\begin{remark}
		For a boundary-touching cell, with the facet $\Gamma_{ij} \subset \partial \Omega$, an additional term must be considered:
		\begin{equation}
			\nabla p(\xx_i) = \sgrad{p}{i} + \frac{1}{|\omega_i|} \int_{\Gamma_{ij}}  (\xx - \xx_i) \; \left(\nabla p(\xx_i) \cdot \bm{n} \right) \, \dd{S} + O(\diam(\omega_i)).
		\end{equation}
		For a Robin boundary condition of the type
		\begin{equation}
			ap + b\pdv{p}{n} = 0,
		\end{equation}
		it is reasonable to approximate the gradient as 
		\begin{equation}
			\nabla p(\xx_i) = \sgrad{p}{i} - \frac{|\Gamma_i|}{|\omega_i|} \frac{ap_i}{b} (\bm{m}_{ij} - \xx_i) + O(\diam(\omega_i)).
			\label{eq:grad_bdry}
		\end{equation}
		Note that the extra term in \eqref{eq:grad_bdry} disappears for a homogeneous Neumann condition.
	\end{remark}
	
	\bigskip
	Next, we consider a different discretization of the gradient operator. This takes the form
	\begin{equation}
		\wgrad{p}{i} = \frac{1}{|\omega_i|} \sum_j \frac{|\Gamma_{ij}|}{r_{ij}}p_{ij}(\bm{m}_{ij} - \xx_j),
		\label{eq:wgrad}
	\end{equation}
	which is motivated by the following ``discrete integration by parts''.
	\begin{theorem} Let $p(\xx)\in \mathcal{C}^2$ and $q(\xx)\in \mathcal{C}^2$, and let us consider the discrete gradient operators \eqref{eq:sgrad} and \eqref{eq:wgrad}. The following identity holds true for $i=1,2,\ldots,N$:
		\begin{equation}
			\sum_i |\omega_i| (q_i \sgrad{p}{i} + p_i \wgrad{q}{i}) = \sum_{i} p_{i} q_{i} \bm{S}_i.
			\label{eq:discrete_ibp}
		\end{equation}  
	\end{theorem}
	\begin{proof}
		After some algebra, we find:
		\begin{equation}
			\begin{split}
				&\sum_i |\omega_i| (q_i \sgrad{p}{i} + p_i \wgrad{q}{i}) = \\
				=&\sum_{i} \sum_{j} \frac{|\Gamma_{ij}|}{r_{ij}} p_i q_i \xx_{ij}
				+ \sum_{i} \sum_{j} |\Gamma_{ij}| \bigg( (\bm{m}_{ij} - \xx_i) q_i p_{j} - (\bm{m}_{ij} - \xx_j) p_i q_{j} \bigg).
			\end{split}
		\end{equation}
		The second term on the right-hand side vanishes by virtue of anti-symmetry, whereas the first one can be simplified using the definition \eqref{eqn.dS}, hence obtaining
		\begin{equation}
			\sum_{i} \sum_{j} \frac{|\Gamma_{ij}|}{r_{ij}} p_i q_i \xx_{ij} = -\sum_{i} \sum_{j} |\Gamma_{ij}|p_i q_i \bm{n}_{ij} = \sum_i p_i q_i \bm{S}_i.
		\end{equation}
	\end{proof}
	We shall refer to \eqref{eq:sgrad} and \eqref{eq:wgrad} as \emph{strong} and \textit{weak} gradient, respectively. Unfortunately, our investigation reveals that the weak gradient is no longer first-order exact, but it exhibits an oscillatory nature and cannot be used to obtain reliable point-wise approximations of the continuous gradient operator, especially on highly irregular meshes. See Figure \ref{fig:gradients} for a comparison between the two gradient operators. The situation is similar to SPH, where a pair of gradient operators is defined, such that they satisfy a discrete integration by parts. Enforcing some degrees of exactness on one operator does not carry over to the other one \cite{violeau2016smoothed}.
	\begin{figure}[ht!]
		\centering
		\begin{subfigure}{0.25\linewidth}
			\centering
			\includegraphics[width=\linewidth]{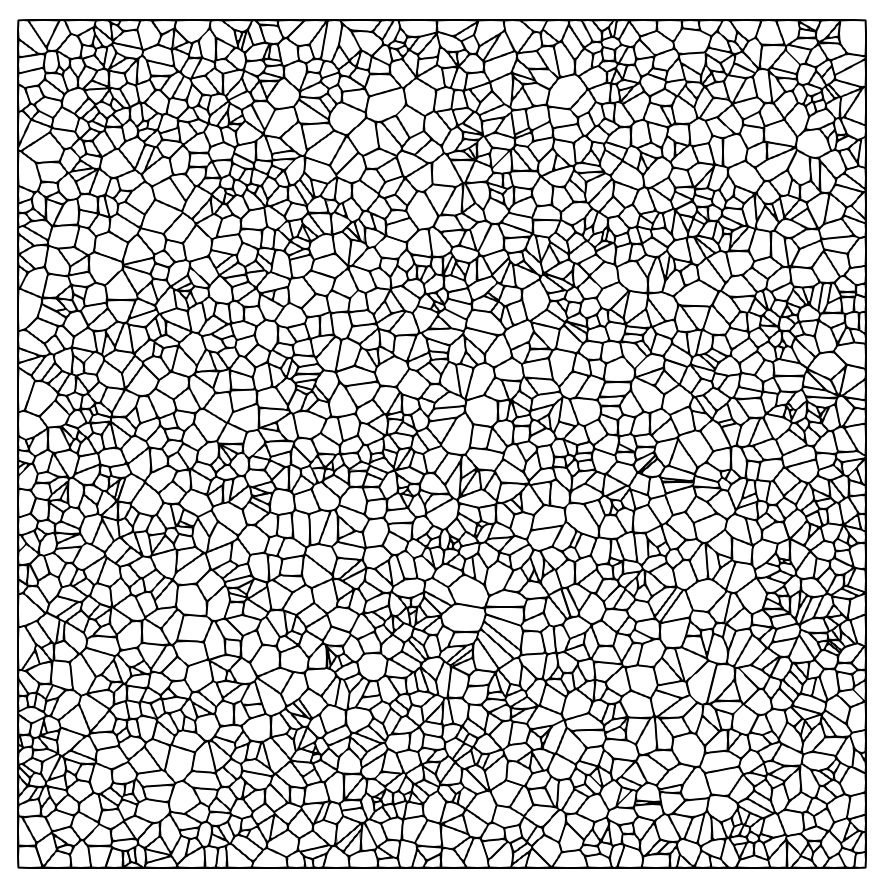}
			\caption{mesh}
		\end{subfigure}%
		\begin{subfigure}{0.25\linewidth}
			\centering
			\includegraphics[width=\linewidth]{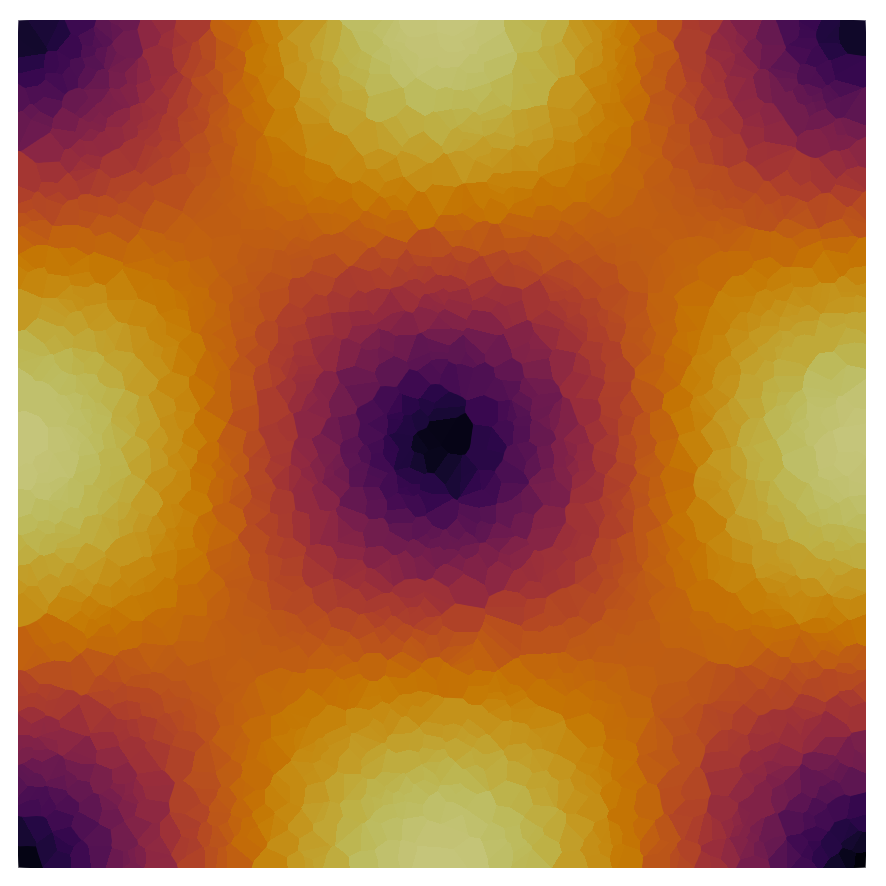}
			\caption{exact}
		\end{subfigure}%
		\begin{subfigure}{0.25\linewidth}
			\centering
			\includegraphics[width=\linewidth]{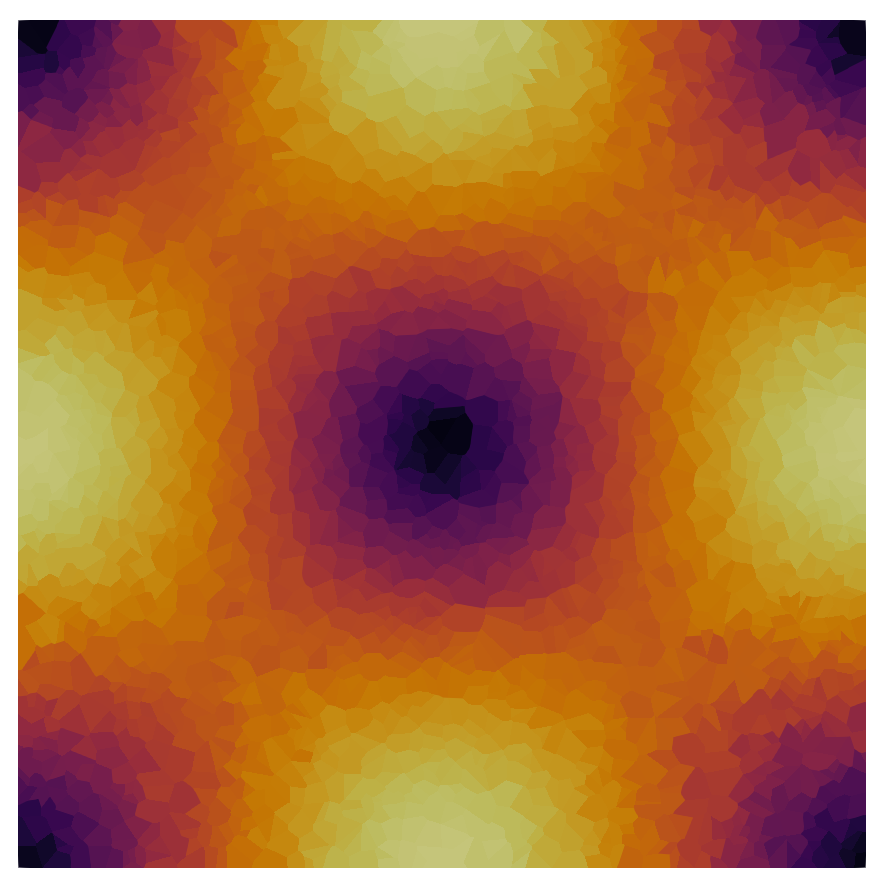}
			\caption{strong}
		\end{subfigure}%
		\begin{subfigure}{0.25\linewidth}
			\centering
			\includegraphics[width=\linewidth]{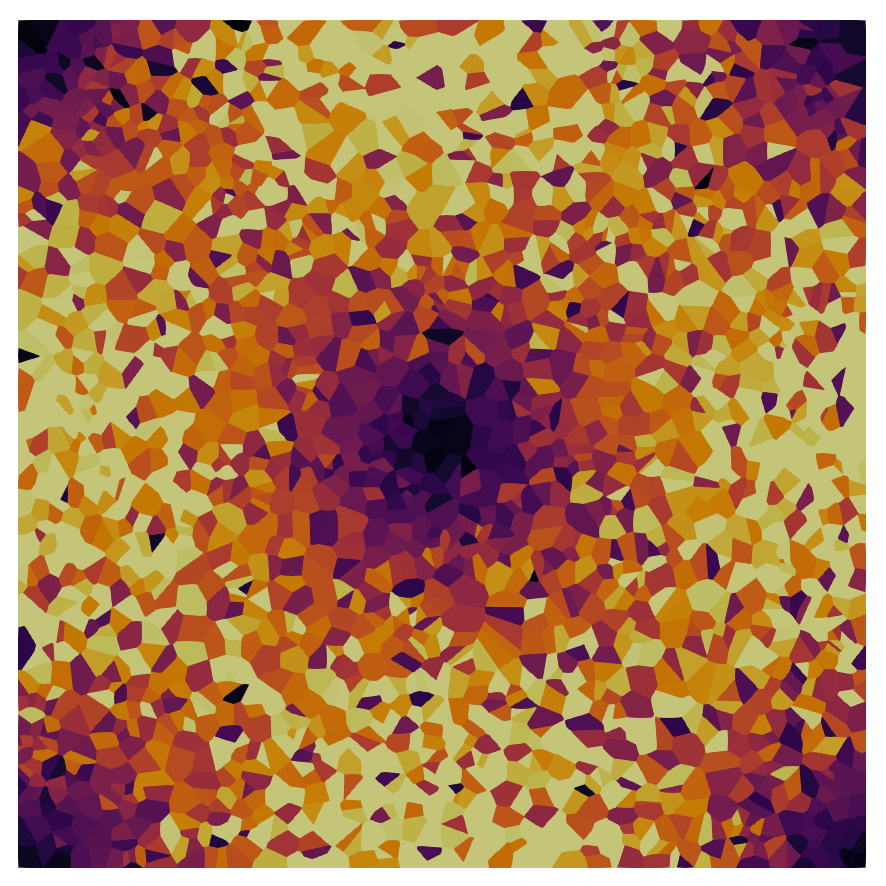}
			\caption{weak}
		\end{subfigure}
		\includegraphics[width=0.3\linewidth]{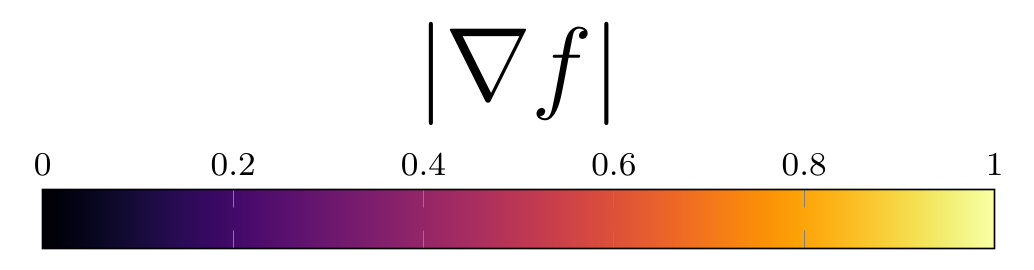}
		\caption{The magnitude of $\nabla f$ on a square $\Omega=[0, 1]^2$ with an irregular Voronoi grid of 6400 cells (a) for $f = \frac{1}{\pi} \cos (\pi x) \cos (\pi y) $. The image (b) and (c) show exact solution and the strong gradient \eqref{eq:sgrad}, respectively, and they are barely distinguishable. The image (d) uses the weak gradient \eqref{eq:wgrad} and is infested with spurious noise.}
		\label{fig:gradients}
	\end{figure}
	
	Nonetheless, the weak gradient is useful to derive the variation of the cell volume $\delta |\omega_i|$. By virtue of \eqref{eq:Omega_wrt_j} and \eqref{eq:Omega_wrt_i}, we get
	\begin{equation}
		\begin{split}
			&\delta |\omega_i| = \sum_j \pdv{|\omega_i|}{\xx_j} \cdot \delta \xx_j = \sum_{j \neq i} \pdv{|\omega_i|}{\xx_j} \cdot \delta \xx_j - \sum_{j\neq i} \pdv{|\omega_i|}{\xx_j} \cdot \delta \xx_i - \bm{S}_i \cdot \delta \xx_i \\
			&= |\omega_i| \, \wdiv{\delta \xx}{i} - \bm{S}_i \cdot \delta \xx_i.
		\end{split}
		\label{eq:cell_diff}
	\end{equation}
	
	\subsection{Semi-discrete space discretization}
	We consider the seeds of a moving Voronoi mesh $\xx_i$ being the material particles with point-values of velocity $\vv_i$, density $\rho_i$ and pressure $p_i$. The mathematical model is given by the velocity equation \eqref{eqn.pde_v}, the divergence constraint in weak form \eqref{eqn.pde_divv}, and the trajectory equation \eqref{eqn.pde_x}. For the first exposition of the method we shall ignore the viscous terms and only include them later. 
	
	For each cell $\omega_i$, a natural semi-discrete space discretization writes
	\begin{subequations}
		\begin{align}
			&	\frac{\text{d}\vv_i}{\text{d}t} = -\frac{1}{\rho_i} \sgrad{p}{i} \label{eqn.bomom_sd} \\
			&\sum_i |\omega_i| \, \vv_i \cdot \sgrad{\varphi}{i} = 0, \quad \forall \, \varphi(\xx), \label{eqn.divv_sd} \\
			&	\frac{\text{d}\xx_i}{\text{d}t} = \vv_i. \label{eqn.x_sd}
		\end{align}
		\label{eq:discrete_sys}
	\end{subequations}
	
	\begin{theorem}\label{th.omega_E}
		Any solution of the semi-discrete scheme \eqref{eq:discrete_sys} satisfies the following properties.
		\begin{itemize}
			\item The individual cell volumes $|\omega_i|$ are conserved:
			\begin{equation}
				\frac{\text{d}|\omega_i|}{\text{d}t} = 0. \label{eqn.domega}
			\end{equation}
			\item The energy
			\begin{equation}
				E = \frac{1}{2} \sum_i |\omega_i| \, \rho_i \, \Vert\vv_i\Vert^2
				\label{eqn.Etot}
			\end{equation}
			is conserved.
		\end{itemize}
	\end{theorem}
	\begin{proof}
		Using the variation \eqref{eq:cell_diff}, the time evolution of $\omega_i$ is given by
		\begin{equation}
			\frac{\text{d}|\omega_i|}{\text{d}t} = |\omega_i|\wdiv{\vv}{i} - \bm{S}_i \cdot \vv_i.
		\end{equation}
		Multiplication by a smooth test function $\varphi(\xx)$ and integration over all cells, leads to 
		\begin{equation}
			\sum_i \frac{\text{d}|\omega_i|}{\text{d}t} \, \varphi_i = \sum_i |\omega_i| \wdiv{\vv}{i} \varphi_i - \sum_i \varphi_i \vv_i \cdot \bm{S}_i.
		\end{equation}
		Using the identity \eqref{eq:discrete_ibp} and the semi-discrete divergence constraint \eqref{eqn.divv_sd}, from the above relation we find that
		\begin{equation}
			\sum_i \frac{\text{d}|\omega_i|}{\text{d}t} \, \varphi_i = -\sum_i |\omega_i| \vv_i \cdot \sgrad{\varphi}{i} = 0, 
		\end{equation}
		which implies \eqref{eqn.domega} after localization.
		
		In light of the previous result we deduce that the cell volume does not depend on time, thus the time evolution of the total energy \eqref{eqn.Etot} yields
		\begin{equation}
			\frac{\text{d} E }{\text{d}t} =\sum_i |\omega_i| \, \rho_i \vv_i \cdot \frac{\text{d} \vv_i}{\text{d}t}.
		\end{equation}
		Inserting the balance of momentum \eqref{eqn.bomom_sd} and using the incompressibility constraint \eqref{eqn.divv_sd}, we obtain
		\begin{equation}
			\frac{\text{d} E }{\text{d}t} = - \sum_i |\omega_i| \, \vv_i \cdot \sgrad{p}{i} = 0.
		\end{equation}
	\end{proof}

	\subsection{Semi-implicit time discretization}
	The idea behind the semi-implicit time discretization for incompressible fluid is fairly standard, although it is here applied for the first time in the context of Lagrangian Voronoi Methods. It can be understood as Chorin iterations \cite{chorin1997}, or equivalently, the projection of the velocity field onto the solenoidal space in the sense of Helmholtz decomposition. This requires the solution of an implicit system on the pressure, globally defined on the entire computational domain, which is obtained by combining the velocity equation \eqref{eqn.pde_v} with divergence-free constraint \eqref{eqn.pde_divv}. On the other hand, the mesh motion is explicit, using the trajectory equation \eqref{eqn.pde_x}.
	
	Let us consider the time interval $T=[0,t_f]$, with $t\in T$ and $t_f$ being the final time of the simulation. The time interval is discretized by a sequence of time steps $\Delta t=t^{n+1}-t^n$, with $t^n$ denoting the current time. We proceed by introducing a semi-implicit time discretization for the semi-discrete scheme \eqref{eq:discrete_sys}.
	
	The Voronoi seed positions are advected explicitly by the velocity field, that is
	\begin{equation}
		\xx_i^{n+1} = \xx_i^{n} + \Delta t \, \vv_i^n,
    \label{eqn.x_fd}
	\end{equation}
	and the Voronoi mesh is regenerated at the new time level $t^{n+1}$, hence obtaining the new tessellation. All geometry related quantities can now be computed, e.g. facet lengths $|\Gamma_{ij}|^{n+1}$, inter-seed distances $r_{ij}^{n+1}$ and facet midpoints $\bm{m}_{ij}^{n+1}$. As proven in Theorem \ref{th.omega_E}, cell volumes are instead constructed to be constant in time, hence we simply have $|\omega_i|^{n+1}=|\omega_i|^{n}=|\omega_i|$.
	
	Even though this step does not modify the velocity $\vv_i^n$ \textit{per se}, it introduces a non-vanishing divergence of the velocity field through the deformation of the mesh. Therefore, the velocity needs to be updated at the next time level by the \textit{correct} pressure gradient, which can be interpreted as a Lagrangian multiplier enforcing the incompressibility constraint. To this aim, we impose that the new velocity $\vv_i^{n+1}$ and pressure $p_i^{n+1}$ satisfy
	\begin{subequations}
		\begin{align}
			&\vv_i^{n+1} = \vv_i^{n} - \frac{\Delta t}{\rho_i} \sgrad{p}{i}^{n+1}, \label{eqn.v_fd}\\
			&\sum_i |\omega_i| \, \vv_i^{n+1} \cdot \sgrad{\varphi}{i} = 0, \quad \forall \, \varphi(\xx^{n+1}). \label{eqn.divv_fd}
		\end{align}
	\label{eqn.P_system}
	\end{subequations}
	Formal substitution of the velocity equation \eqref{eqn.v_fd} into the divergence-free constraint \eqref{eqn.divv_fd}, yields an elliptic problem for the pressure
	\begin{equation}
		\sum_i \frac{|\omega_i|}{\rho_i} \sgrad{p}{i}^{n+1} \cdot \sgrad{\varphi}{i} = \frac{1}{\Delta t} \sum_i |\omega_i| \, \vv_i^{n} \cdot \sgrad{\varphi}{i}.
		\label{eq:discrete_poi}
	\end{equation}
	The associated $N\times N$ system matrix $\A$ is defined for any function $p$ and $\varphi$ to satisfy the following relation:
	\begin{equation}
		\sum_i \frac{|\omega_i|}{\rho_i} \sgrad{p}{i}^{n+1} \cdot \sgrad{\varphi}{i} = \sum_{i} \sum_{j} \varphi_i \, \A_{ij} \, p_j^{n+1}, \label{eq:matA}
	\end{equation}
	thus the matrix $\A$ is clearly symmetric. The right hand side of \eqref{eq:discrete_poi} is an $N$-element vector representing the linear functional
	\begin{eqnarray}
		\sum_i \b_i \, \varphi_i &=& \frac{1}{\Delta t} \sum_i |\omega_i| \, \vv_i^{n} \cdot \sgrad{\varphi}{i} \nonumber \\
		&=& \frac{|\omega_i|}{\Delta t} \wdiv{\vv}{i}^n \label{eq:vecb}.
	\end{eqnarray}
	The coefficients of the system can be evaluated by setting $\varphi_i = \delta_{ik}$ and $p_j = \delta_{j\ell}$ for fixed indexes $(k,\ell)$. Unfortunately, the matrix is not very sparse because $\A_{ij} \neq 0$ not only when $i,j$ are neighbors, but also when $j$ is a neighbor of a neighbor of $i$.  
	
	In the case when the fluid is incompressible and homogeneous ($\rho_i = \rho = const$), it is possible to sparsify the problem \eqref{eq:matA} through the following ``approximation of approximation'', in a direct analogy to incompressible SPH. Recalling the strong gradient definition \eqref{eq:sgrad} and using the homogeneous condition $\rho_i=\rho$, the system matrix in \eqref{eq:matA} is rewritten as follows:
	\begin{eqnarray}
		\sum_i \frac{|\omega_i|}{\rho_i} \sgrad{p}{i}^{n+1} \cdot \sgrad{\varphi}{i} &=& -\frac{1}{\rho} \sum_{i} \sum_{j} \frac{|\Gamma_{ij}|^{n+1}}{ r_{ij}^{n+1}} \, p_{ij}^{n+1} \, (\bm{m}_{ij}^{n+1} - \xx_{i}) \cdot \sgrad{\varphi}{i} \nonumber \\
		&\approx& -\frac{1}{\rho}\sum_{i} \sum_{j} \frac{|\Gamma_{ij}|^{n+1}}{r_{ij}^{n+1}} \, p_{ij}^{n+1} \, (\varphi(\bm{m}_{ij}^{n+1}) - \varphi_i) \nonumber \\
		&=& \frac{1}{\rho} \sum_{i} \sum_{j} \frac{|\Gamma_{ij}|^{n+1}}{r_{ij}^{n+1}} \, p_{ij}^{n+1} \varphi_i.
		\label{eq:approxapprox}
	\end{eqnarray}
	The term involving $\varphi(\bm{m}_{ij}^{n+1})$ disappears because of the anti-symmetry of the addends. The new matrix $\B_{ij}$ is thus defined such as
	\begin{equation}
		\sum_j \B_{ij} p_j^{n+1} = \frac{1}{\rho}\sum_{j} \frac{|\Gamma_{ij}|^{n+1}}{r_{ij}^{n+1}} \, p_{ij}^{n+1} \label{eq:matB},
	\end{equation} 
	and it is much simpler, more sparse and still symmetric compared to the matrix $\A$ in \eqref{eq:matA}. Figure \ref{fig:sparsity} depicts an illustration of the sparsity of $\B$ when compared to incompressible SPH. SILVA produces much sparser matrices, which is a great advantage compared to SPH-based methods: for SILVA, the number of non-zeros per node is less than 7, but for SPH it is about 24. We expect that the difference becomes even more pronounced in three space dimensions.
	\begin{figure}[ht!]
		\centering
		\begin{subfigure}{0.5\linewidth}
			\includegraphics[width=\linewidth]{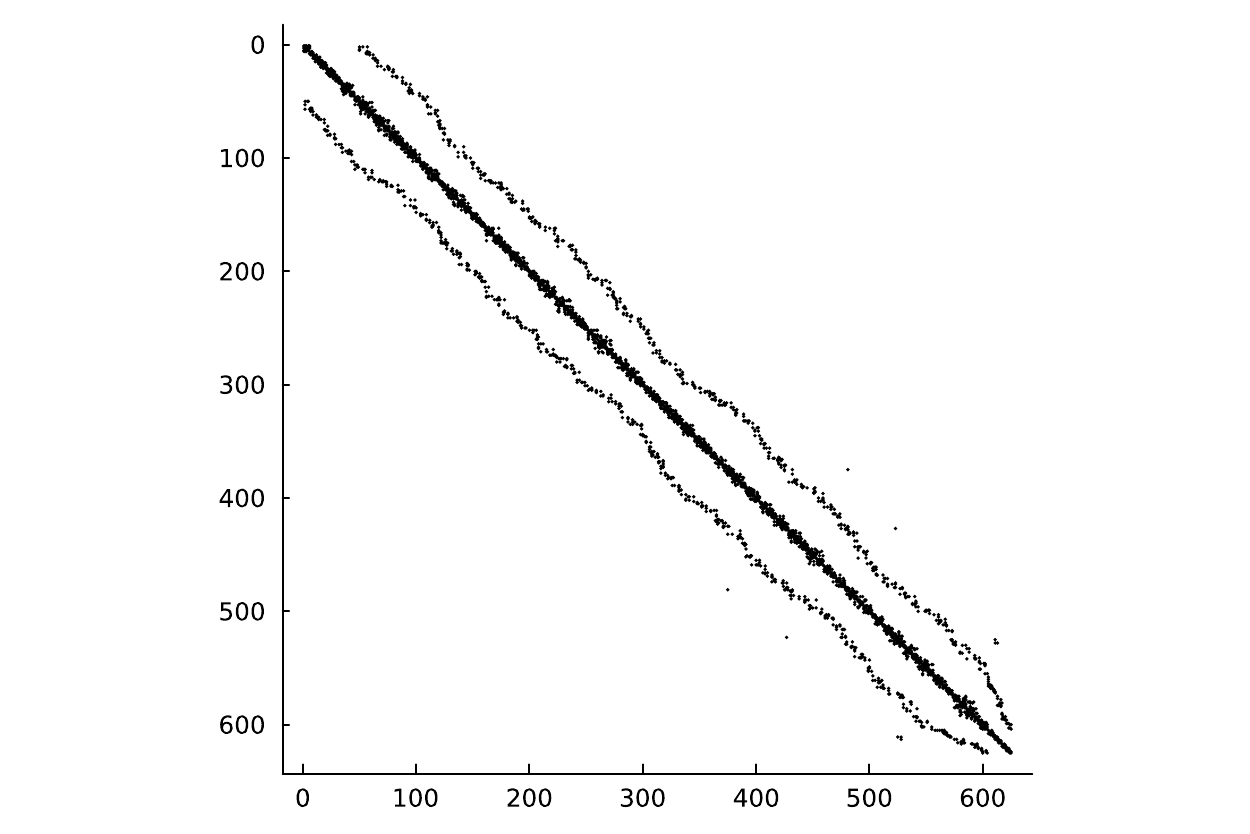}
			\subcaption{SILVA}
		\end{subfigure}%
		\begin{subfigure}{0.5\linewidth}
			\includegraphics[width=\linewidth]{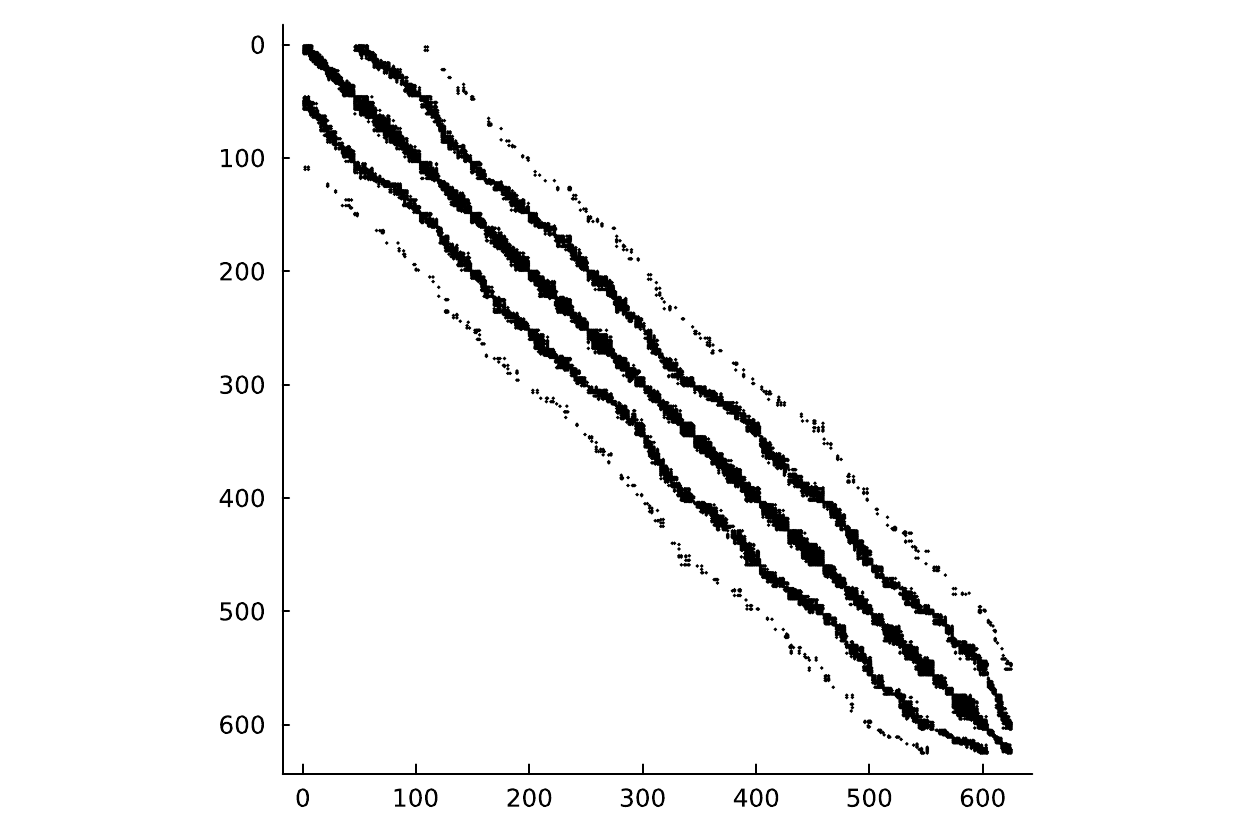}
			\subcaption{ISPH}
		\end{subfigure}
		\caption{Example of a sparsity pattern of the left-hand-side matrix in SILVA and ISPH on uniformly random distribution of 625 points (Voronoi seeds) in two dimensions. The smoothing length for ISPH is $1.5\delta r$ and the Wendland quintic kernel is used. For the sake of readability, the points are sorted by their occurrence in the cell list.}
		\label{fig:sparsity}
	\end{figure}
	In the definition \eqref{eq:matB}, we can recognize the approximation of the Laplacian operator as used in the Finite Volume Method on Voronoi meshes, which generalizes the standard central finite difference operator on Cartesian meshes \cite{eymard2000finite}. Since
	\begin{equation}
		\sum_{i} \sum_{j} p_i^{n+1} \, \B_{ij} \, p_j^{n+1} = \frac{1}{2\rho}\, \sum_{i} \sum_{j} \frac{|\Gamma_{ij}|^{n+1}}{r_{ij}^{n+1}}(p_{ij}^{n+1})^2,
	\end{equation} 
	we see that the matrix $\B$ is positive semi-definite and singular, with kernel generated by a constant vector $p_i^{n+1} \equiv 1$ (since $\Omega$ is connected). This implies that an additive constant on the pressure must be added to solve the Poisson problem with Neumann boundary conditions. Luckily, when we use a Krylov solver, like the Conjugate Gradient (CG) or the minimal residual (MINRES) method, this never becomes a real issue, since the matrix $\B$ is regular in the invariant subspace of vectors with zero average. The right hand side $\b$ defined by \eqref{eq:vecb} clearly belongs to this vector space, since
	\begin{equation}
		\sum_i \b_i = \frac{1}{\Delta t} \sum_i |\omega_i| \, \vv_i^{n} \cdot \sgrad{1}{i} = 0.
	\end{equation}

  Once the new pressure field $p_i^{n+1}$ is determined, the divergence-free velocity field $\vv_i^{n+1}$ is readily updated by means of \eqref{eqn.v_fd}.
  
  \begin{remark}
  	As an alternative to \eqref{eq:approxapprox} it could be possible to obtain a sparse matrix by adopting the staggered grid approach, along the lines of the semi-implicit finite volume schemes proposed in \cite{Casulli1990}. We leave this option for future investigations.
  \end{remark} 
	
	\subsection{Vortex core instability: stabilized pressure gradient}
	\label{sec:stabilization}
	In the SPH method, a well-known numerical artifact, called \textit{tensile instability}, appears in regions with negative pressure and leads to undesirable clustering of particles \cite{monaghan2000sph}. In a naive implementation of a Lagrangian Voronoi method, a similar problem arises. Even if every sample point occupies the same area in the sense of the Voronoi mesh, this does not forbid the seeds to approach each other arbitrarily close, leading to stability issues. This problem is particularly evident near pressure minima (such as vortex cores). To shed some light on this problem, let us suppose that close to a seed $\xx_i$, the pressure is given by
	\begin{equation}
		p = \lambda \, |\xx - \xx_i|^2,
		\label{eqn.p2}
	\end{equation}
	where $\lambda$ is a positive constant. Then $\nabla p (\xx_i) = 0$ and the particle should not be accelerated. Unfortunately, this is not what we observe because at the discrete level, the strong gradient operator \eqref{eq:sgrad} applied to the above pressure filed yields
	\begin{equation}
		\begin{split}
			\sgrad{p}{i} = \lambda \sgrad{|\bm{x} - \xx_i|^2}{i} = \frac{\lambda}{|\omega_i|} \sum_j |\Gamma_{ij}| \, r_{ij} \, (\bm{m}_{ij} - \xx_i) \neq 0. \label{eq:clever_zero}
		\end{split}
	\end{equation}
	Indeed, a Voronoi cell $\omega_i$ can be divided into sub-cells (triangles for $d=2$), each of which is a convex hull of $\{\xx_i\}\cup \Gamma_{ij}$ and has volume
	\begin{equation}
		V_{ij} = \frac{1}{d} |\Gamma_{ij}| r_{ij} 
		\label{eq:subvolume}
	\end{equation}
	and a centroid
	\begin{equation}
		\bm{c}_{ij} = \frac{d}{d+1}\bm{m}_{ij} + \frac{1}{d+1}\xx_{i}.
		\label{eq:centroid}
	\end{equation}
	Substitution of \eqref{eq:subvolume} and \eqref{eq:centroid} into \eqref{eq:clever_zero} leads to
	\begin{equation}
			\sgrad{p}{i} 
			=\frac{\lambda(d+1)}{|\omega_i|} \sum_j V_{ij} \left( \bm{c}_{ij} - \xx_i \right)
			= \lambda(d+1) \left( \bm{c}_{i} - \xx_i \right),
			\label{eqn.gradpL}
	\end{equation}
	where $\bm{c}_{i}$ is the centroid of $\omega_i$. Therefore, unless $\xx_i = \bm{c}_i$, the seed $\xx_i$ will be accelerated away from the centroid, hence distorting the cell. This means that we have to remove the potentially non-zero term given by $\lambda(d+1) \left( \bm{c}_{i} - \xx_i \right)$. To determine the positive coefficient $\lambda$, let us compute the Laplacian of the pressure field \eqref{eqn.p2}, which leads to
	\begin{equation}
		\lambda = \frac{\Delta p(\xx_i)}{2d}.
	\end{equation}
	The above expression is then inserted in \eqref{eqn.gradpL} to obtain the following definition of a \textit{stabilized pressure gradient}:
	\begin{equation}
		\sgrad{p}{i}^\mathrm{s} = \sgrad{p}{i} - \frac{d+1}{2d} [\langle \Delta p \rangle_i]^+\left( \bm{c}_{i} - \xx_i \right), \label{eq:grad_s1}
	\end{equation}
	where $\langle \Delta p \rangle_i$ is evaluated relying on the discrete Laplace operator \eqref{eqn.Laplace_p}, which will be introduced in the next section. Here, $[f]^+=\max(f,0)$ denotes the positive part of a function. We shall use \eqref{eq:grad_s1} instead of $\sgrad{p}{i}$ in \eqref{eqn.v_fd} for the update of the velocity field. Note that this stabilizer does not introduce any numerical parameter nor does it interfere with the exactness of the discrete gradient on linear functions.
	
	Other stabilization techniques for Lagrangian Voronoi Methods are suggested in the literature and can be used also in our setting. In \cite{springel2011hydrodynamic}, the clustering is prevented in the spirit of ALE method, by introducing a velocity of coordinates in the direction defined by the vector $\bm{c}_i - \xx_i$. In \cite{despres2024lagrangian}, a stabilizer is designed which adds a repulsion force nearby particles and is proven to be weakly consistent. We prefer to adopt the stabilizer \eqref{eq:grad_s1} because it is relatively simple to implement and it preserves the fully Lagrangian nature of SILVA.

	\subsection{Discrete viscous operator}
	It remains to discretize the viscous terms in the velocity equation \eqref{eqn.pde_v}. In a finite volume sense, the Laplacian of a twice differentiable vector field can be estimated as follows:
	\begin{eqnarray}
		\Delta \vv(\xx_i) &\approx& \frac{1}{|\omega_i|} \int_{\omega_i} \Delta \vv \, \dd \xx \nonumber \\
		&=& \frac{1}{|\omega_i|} \int_{\partial \omega_i} \nabla \vv \; \bm{n} \, \dd S \nonumber \\
		&\approx&-\frac{1}{|\omega_i|}\sum_j \frac{|\Gamma_{ij}|}{r_{ij}} \vv_{ij} =:\langle \Delta \vv \rangle_i. \label{eqn.Laplace_p}
	\end{eqnarray}
	Using the above definition and assuming a constant kinematic viscosity coefficient $\nu$, the viscous force in the incompressible Navier-Stokes model \eqref{eqn.pde} can be easily implemented within each cell $\omega_i$ as
	\begin{equation}
		\bm{f}_{\mathrm{visc}, i} = \nu \Delta \vv(\xx_i) \approx \nu \langle \Delta \vv \rangle_i = -\frac{\nu}{|\omega_i|}\sum_j \frac{|\Gamma_{ij}|}{r_{ij}} \vv_{ij},
		\label{eq:Viscosity}
	\end{equation}
	with $\vv_{ij}=\vv_i-\vv_j$. Naturally, if viscosity is included, we can no longer expect the conservation of energy $E$ in \eqref{eqn.Etot}, since $E$ represents only the mechanical component of energy and does not take into account entropy production. However, the velocity field is still divergence-free and the cell areas $|\omega_i|$ are still preserved. Regarding the time discretization, we include the viscous force right after the advection step \eqref{eqn.x_fd} in an explicit manner. In this way, we obtain an intermediate velocity field
	\begin{equation}
		\vv^{n*}_i = \vv^{n}_i + \Delta t \, \bm{f}_{\mathrm{visc}, i}(\xx^{n+1}_i, \vv_i^{n}),
	\end{equation}
	which is then projected to a solenoidal space by the implicit pressure step \eqref{eqn.P_system} with $\vv^{n*}$ in place of $\vv^n$.
	
	\begin{remark}
		To prescribe a Dirichlet condition for the velocity vector of the type
		\begin{equation}
			\vv = \vv_D, \quad \xx \in \Gamma_D,
			\label{eq:Dirichlet}
		\end{equation}
		on a linear boundary $\Gamma_D$, we need to take into account the viscous force by which walls act onto Voronoi cells. The force can be introduced by means of a fictitious mirror polygon $\omega_j'$, which is the reflection of $\omega_i$ with respect to $\Gamma_D$ and which has velocity 
		\begin{equation}
			\vv_j' = 2\vv_D - \vv_i.
		\end{equation}
		Including these reflected cells, Equation \eqref{eq:Viscosity} can be used without modification by setting $ \vv_{ij}=\vv_i - \vv_j'$. 
		Following the same reasoning, no-slip boundary conditions are obtained by setting $\vv_j' =- \vv_i$, which means $\vv_D=\mathbf{0}$.
	\end{remark}
	
	\subsection{Summary of the SILVA method}
	Finally, let us briefly summarize what one time step of SILVA looks like.
	\begin{enumerate}
		\item Update of the seed positions $\xx_i^{n+1}$ with \eqref{eqn.x_fd}, hence by solving explicitly the trajectory equation \eqref{eqn.pde_x}.
		\item Re-generation of the Voronoi mesh according to the procedure detailed in Section \ref{sec.VoronoiMesh}. Update of all geometry related quantities (facet lengths, midpoints and inter-seed distances).
		\item Explicit computation of the viscous forces using \eqref{eq:Viscosity}.
		\item Computation of the new pressure $p_i^{n+1}$ by implicitly solving the discrete pressure Poisson equation \eqref{eq:discrete_poi} with the matrix $\B$ given by \eqref{eq:matB} and the right-hand-side vector $\b$ given by \eqref{eq:vecb} using the MINRES method. 
		\item Update of the divergence-free velocity field $\vv_i^{n+1}$ by means of \eqref{eqn.v_fd} with the stabilized pressure gradient \eqref{eq:grad_s1}. 
	\end{enumerate}
	
	\section{Numerical results} \label{sec.test}
	
	\subsection{Taylor-Green Vortex} \label{sec:tagr}
	We consider the Taylor-Green vortex benchmark on the computational domain $\Omega = [-0.5, 0.5]^2$. This test case has the analytical solution in the form of an exponentially decaying vortex with the velocity field given by
	\begin{equation}
		\vv(t, x,y) = \begin{pmatrix}
			\quad \cos \pi x \; \sin \pi y\\
			-\sin \pi x \; \cos \pi y,
		\end{pmatrix} \exp\left(-\frac{2t\pi^2}{\mathrm{Re}}\right),
	\end{equation}
	while the pressure field is
	\begin{equation}
		p(t, x,y) = \frac{\sin^2 \pi x + \sin^2 \pi x - 1}{2}\exp\left(-\frac{2t \pi^2}{\mathrm{Re}}\right) .
	\end{equation}
	The total kinetic energy is
	\begin{equation}
		E = \frac{1}{2} \exp\left(-\frac{2t \pi^2}{\mathrm{Re}}\right).
	\end{equation}
  $\mathrm{Re} = U_0 \, L_0/\nu$ is the Reynolds number and it represents the ratio between inertial and viscous forces in the studied phenomenon. $U_0=1$ and $L_0=1$ are the reference velocity and domain size, respectively.
	A free-slip boundary and a compatible initial condition is prescribed. In accordance with \cite{carlino2024arbitrary}, the 
	simulation is terminated at time $t_f=0.2$, and we measure the energy conservation error as well as the $L^2$ error of 
	pressure, velocity and the divergence of velocity, evaluated using the discrete operator \eqref{eq:wgrad}. The mesh convergence is tested on two types of Voronoi meshes: one is initiated 
	as a structured Cartesian mesh, while the other one as an unstructured mesh, see Figure \ref{fig:tagr_final}. 
	Three different values of Reynolds number are considered, namely $\mathrm{Re}=\{ 400, 1000, \infty\}$, and the obtained convergence rates are depicted in Figure \ref{fig:tagr_convergence}. One can observe 
	that the convergence of the velocity error is super-linear. The divergence of the velocity field is not 
	conserved up to the machine precision, due to the adoption of a sparser matrix and algebraic errors, but has a very fast convergence rate. 
	\begin{figure}[htb!]
		\centering
		\begin{subfigure}{0.5\linewidth}
			\includegraphics[width = \linewidth]{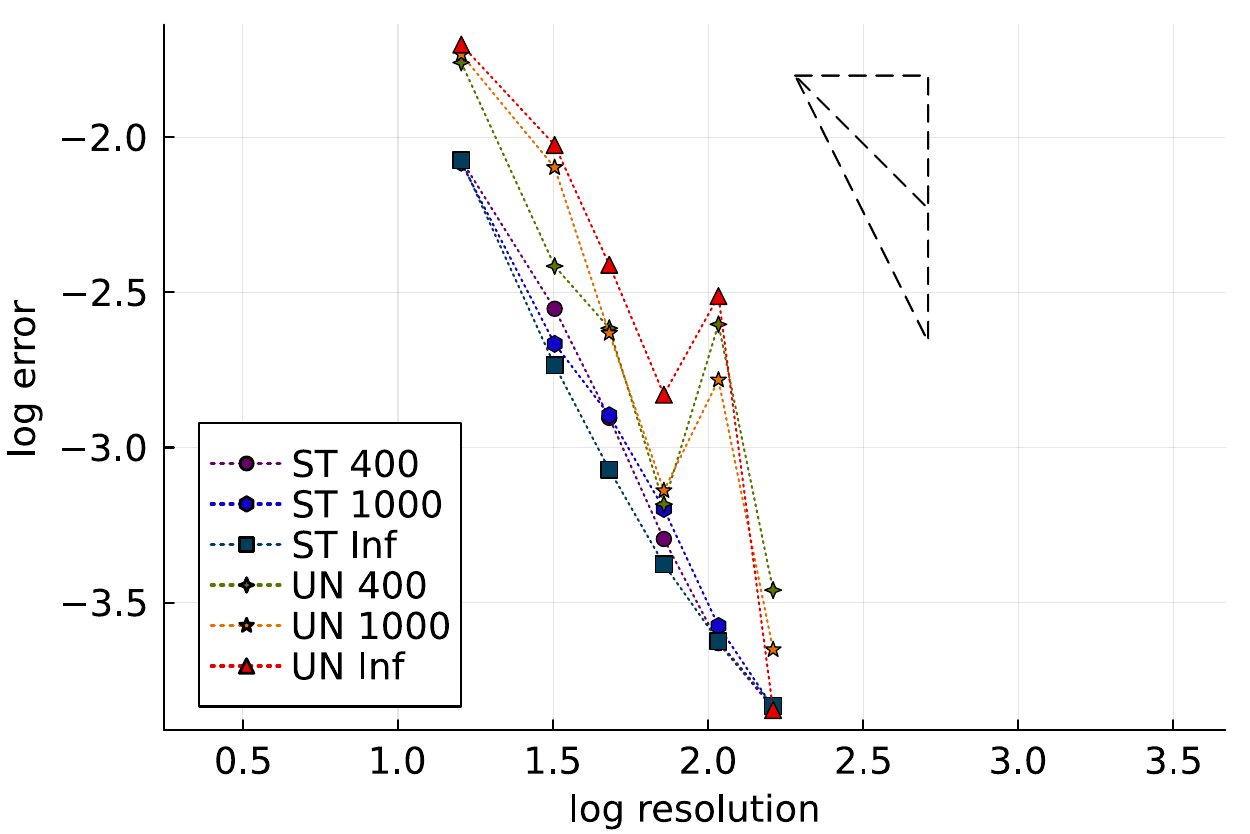}
			\subcaption{velocity error}
		\end{subfigure}%
		\begin{subfigure}{0.5\linewidth}
			\includegraphics[width = \linewidth]{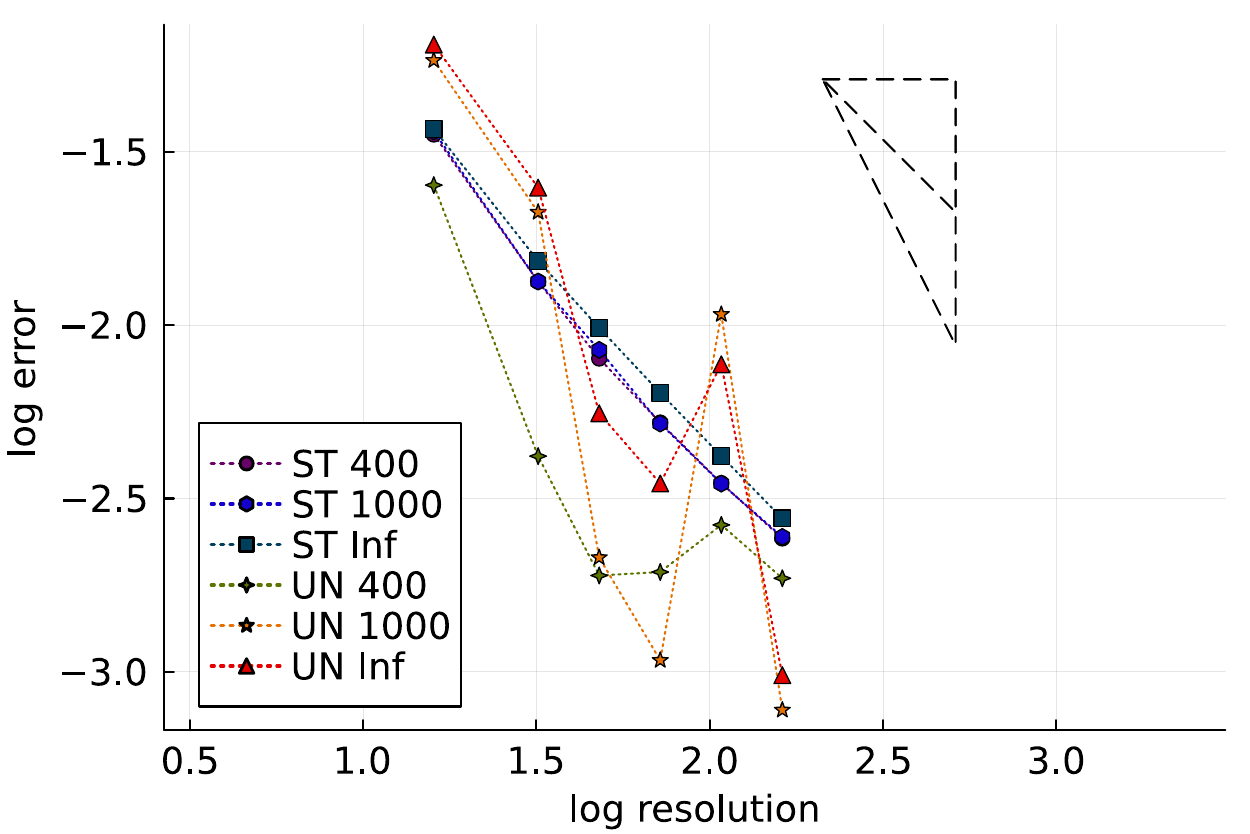}
			\subcaption{pressure error}
		\end{subfigure}
		\begin{subfigure}{0.5\linewidth}
			\includegraphics[width = \linewidth]{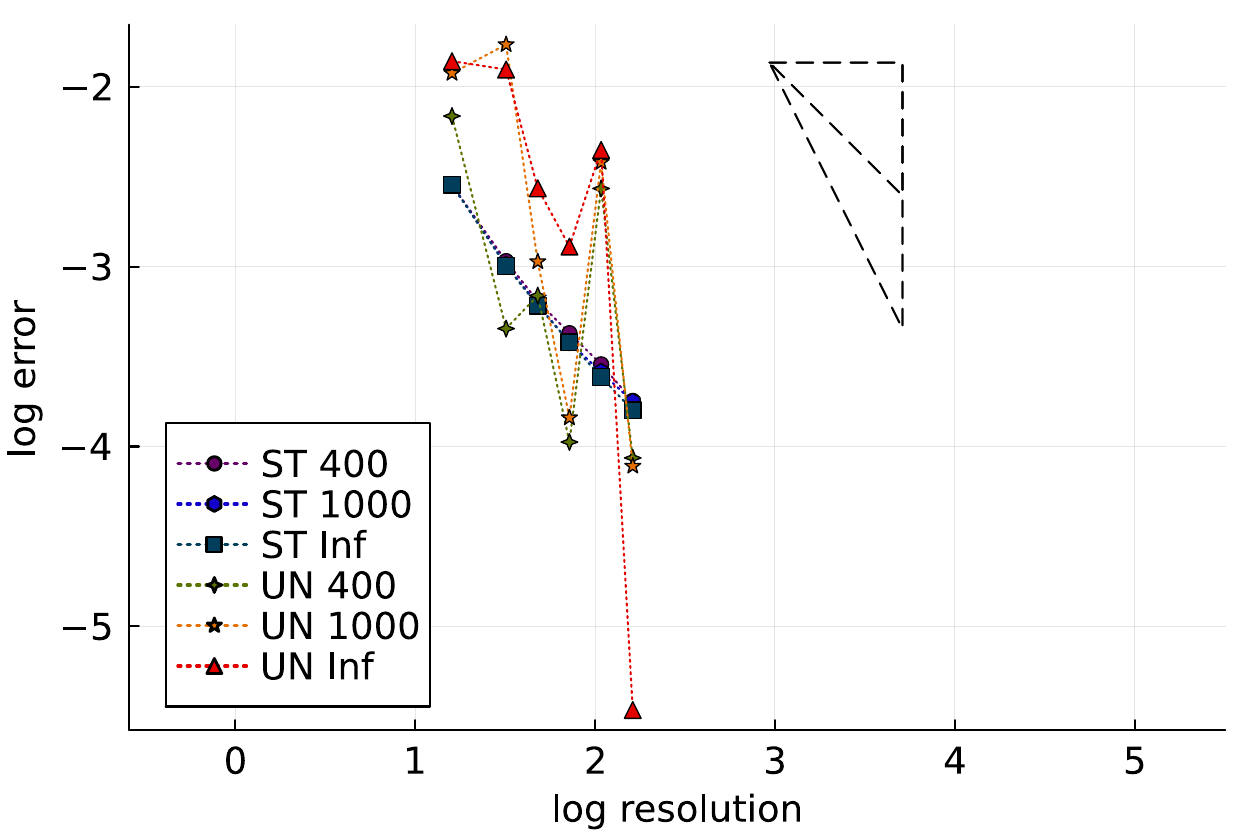}
			\subcaption{energy error}
		\end{subfigure}%
		\begin{subfigure}{0.5\linewidth}
			\includegraphics[width = \linewidth]{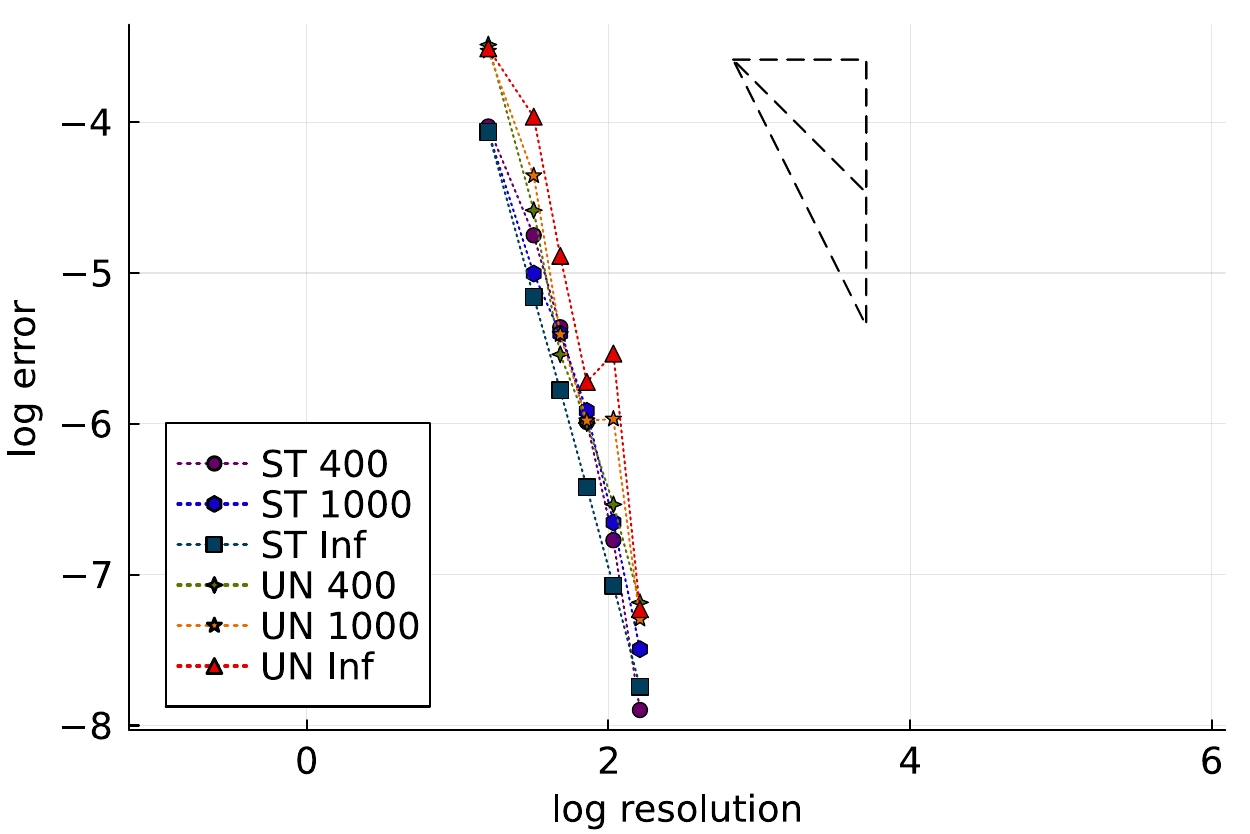}
			\subcaption{divergence error}
		\end{subfigure}
		\caption{Errors of velocity, pressure, energy and velocity divergence in the Taylor-Green vortex benchmark. The horizontal axis corresponds to $\log_{10} N$, where $N$ is the number of cells per side of the computational domain. The vertical axis indicates the logarithm of error. We also compare the result for a structured (ST) rectangular and an unstructured (UN) Vogel grid for $\mathrm{Re}=\{ 400, 1000, \infty\}$. The dashed triangles indicate what the reference linear and quadratic convergence slopes.}
		\label{fig:tagr_convergence}
	\end{figure}

  \begin{figure}[htb!]
  	\centering
  	\begin{subfigure}{0.5\linewidth}
  		\centering
  		\includegraphics[width = 0.7\linewidth]{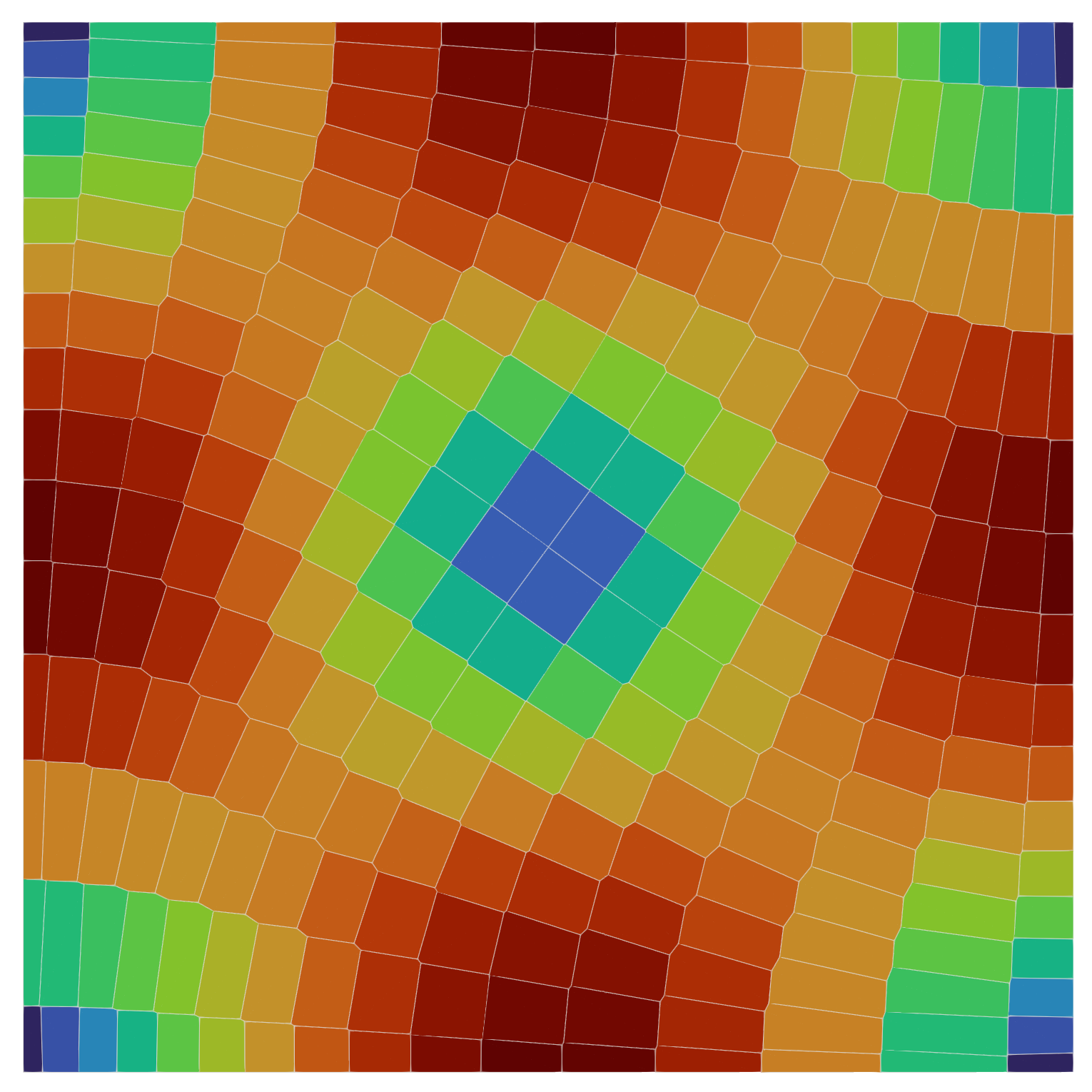}
  		\subcaption{Rectangular grid}
  	\end{subfigure}%
  	\begin{subfigure}{0.5\linewidth}
  		\centering
  		\includegraphics[width = 0.7\linewidth]{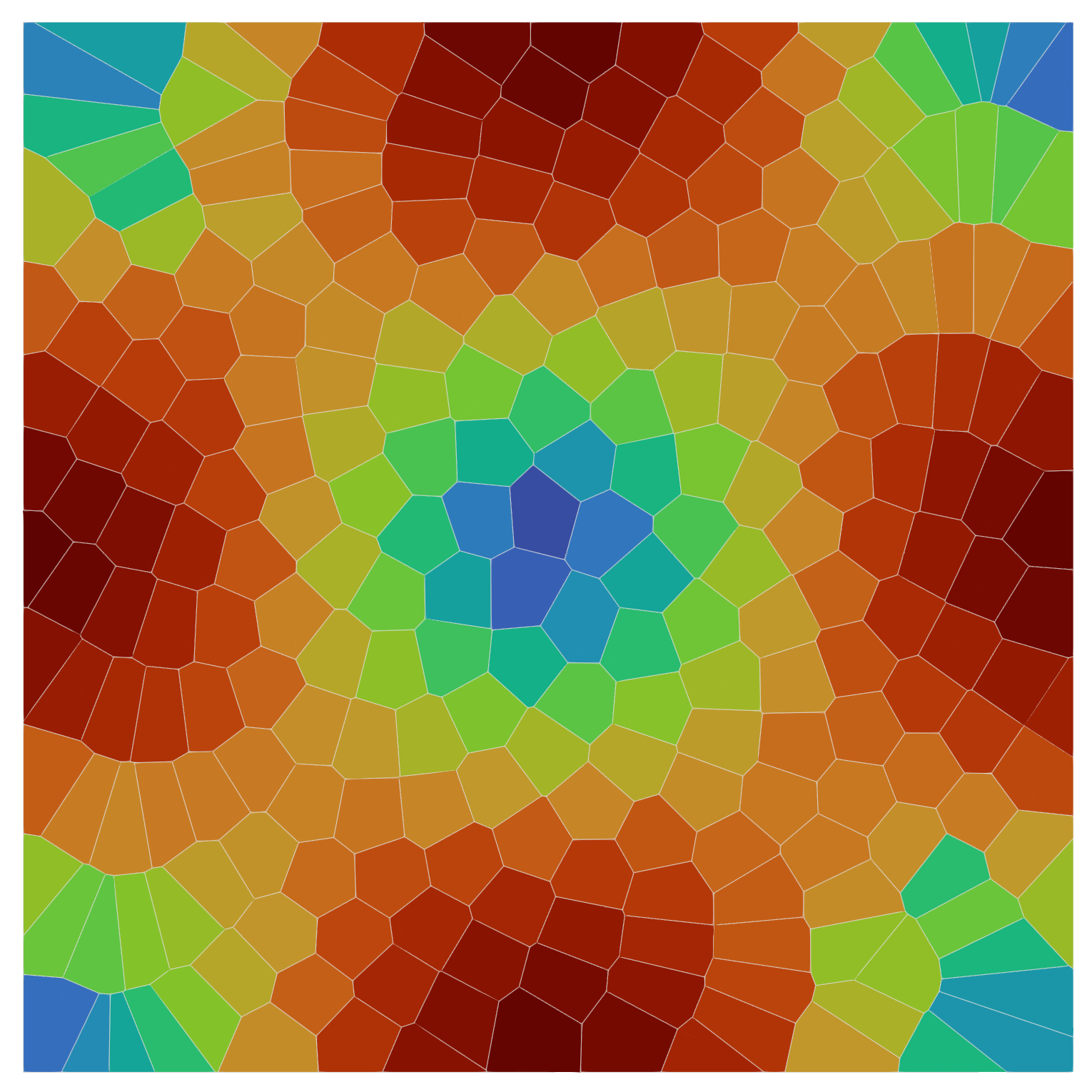}
  		\subcaption{Unstructured grid}
  	\end{subfigure}
  	\includegraphics[width=0.4\linewidth]{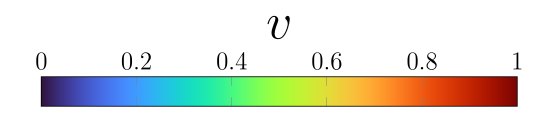}
  	\caption{The final time step of Taylor-Green Vortex for $N = 16$, $\mathrm{Re} = \infty$.}
  	\label{fig:tagr_final}
  \end{figure}

	\subsection{Gresho Vortex}
	The Gresho vortex problem is a standard test case for inviscid fluids. The initial conditions impose a vortex with a triangular velocity profile. With no viscosity, the flow field is stable under evolution, but this behavior is difficult to recover numerically. For the initial setup of this test, we refer to \cite{liska2003comparison}. The computational domain is initially discretized with three different resolutions, namely with $N=50^2$, $N=100^2$ and $N=200^2$ Voronoi seeds. The results of the simulation are depicted in Figure \ref{fig:Gresho_profile}, and they indicate that SILVA preserves the triangular profile with reasonable accuracy, while exhibiting some diffusive behavior. Furthermore, mesh convergence with the respect to the loss of energy is shown. The color map of the velocity field is shown in Figure \ref{fig:Gresho_pics}. Note that the final time $t_f=3$ corresponds to more than three full revolutions of the vortex, which would be clearly impossible to simulate on a Lagrangian mesh with fixed topology. 
	Without using the stabilizer introduced in Section \ref{sec:stabilization}, we observe the well-known vortex core instability, as illustrated in Figure \ref{fig:Gresho_instability}.

	\begin{figure}[ht!]
		\centering
		\includegraphics[width = 0.5\linewidth]
		{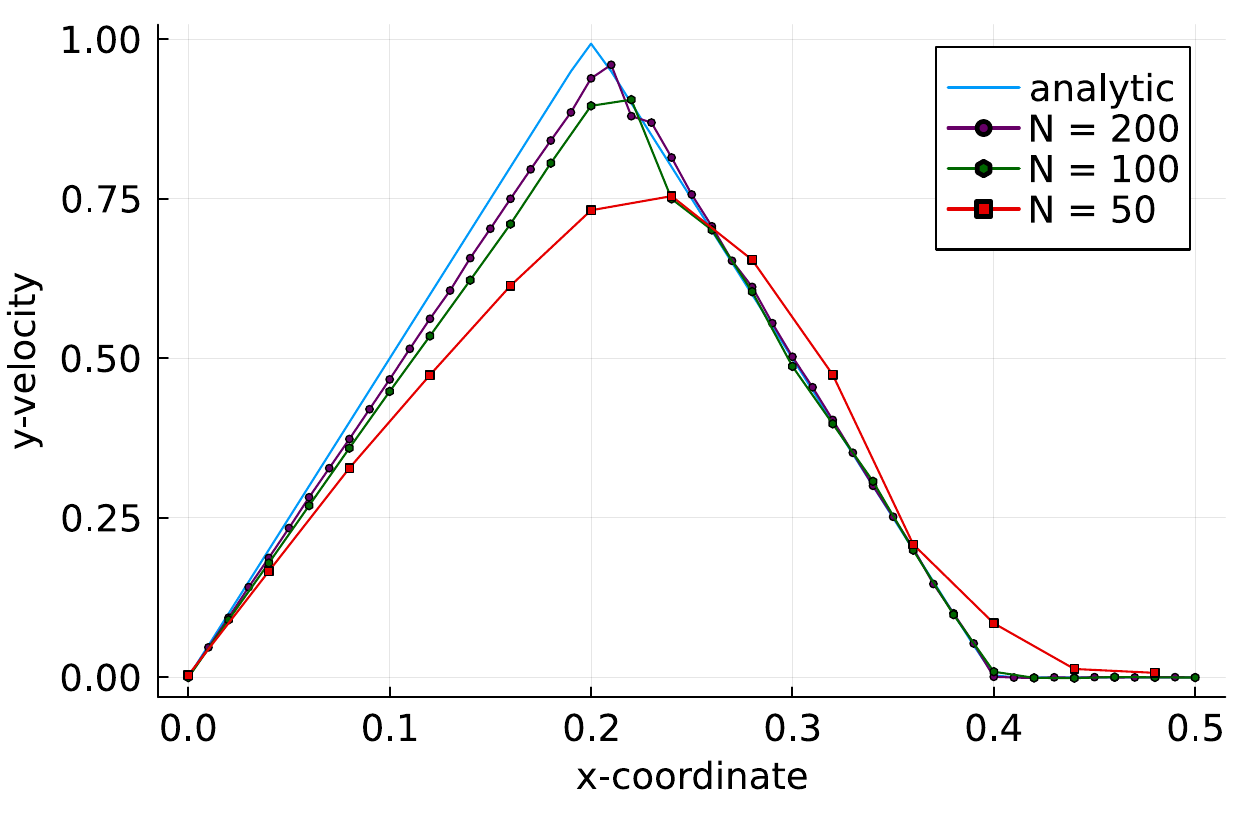}%
		\includegraphics[width = 0.5\linewidth]{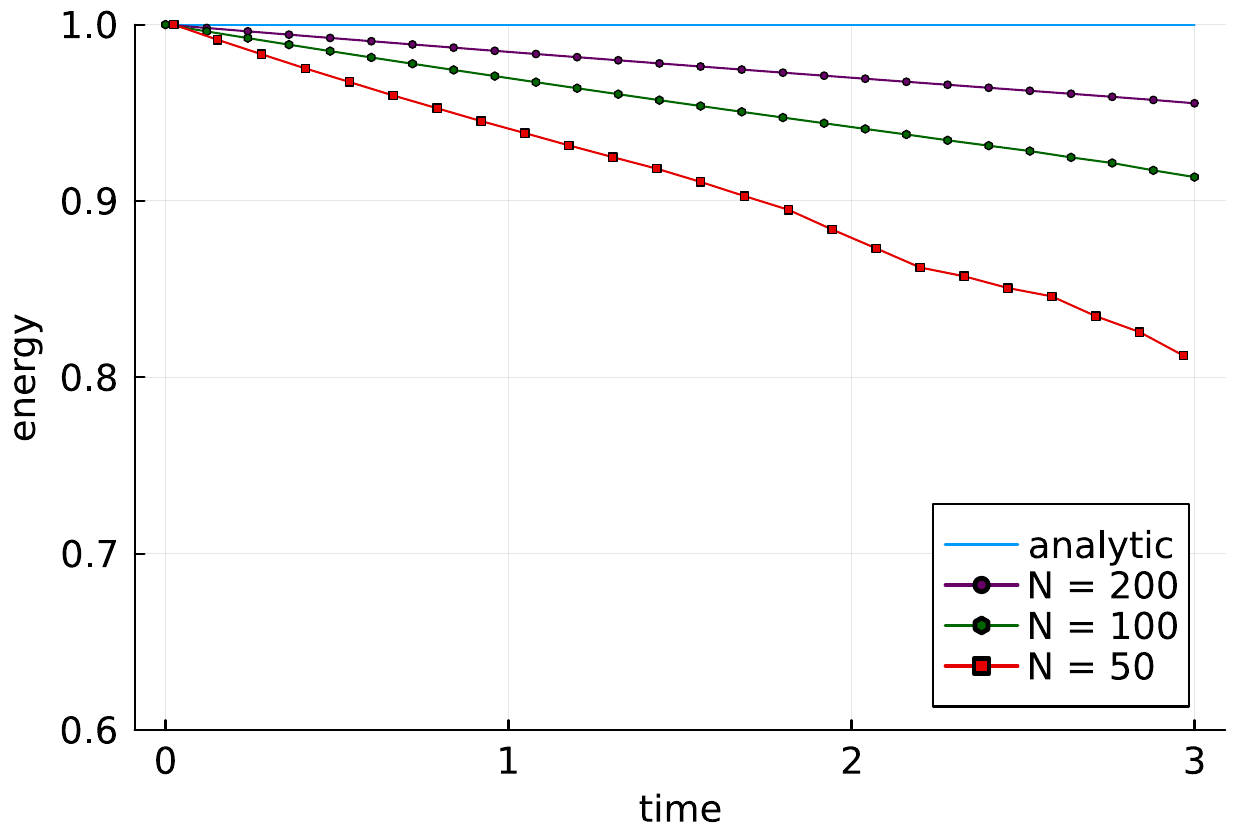}
		\caption{Left: comparison of the tangential velocity along the $x$-axis with the exact solution in the Gresho vortex benchmark. The result is at $t = 3$ and compares three different resolutions and an analytical solution. Right: time evolution of energy (relative to the initial state). The graph can be compared to the high-order semi-implicit method presented in \cite{boscheri2021high}.}
		\label{fig:Gresho_profile}
	\end{figure}
	
	\begin{figure}[ht!]
		\centering
		\includegraphics[width = \linewidth]{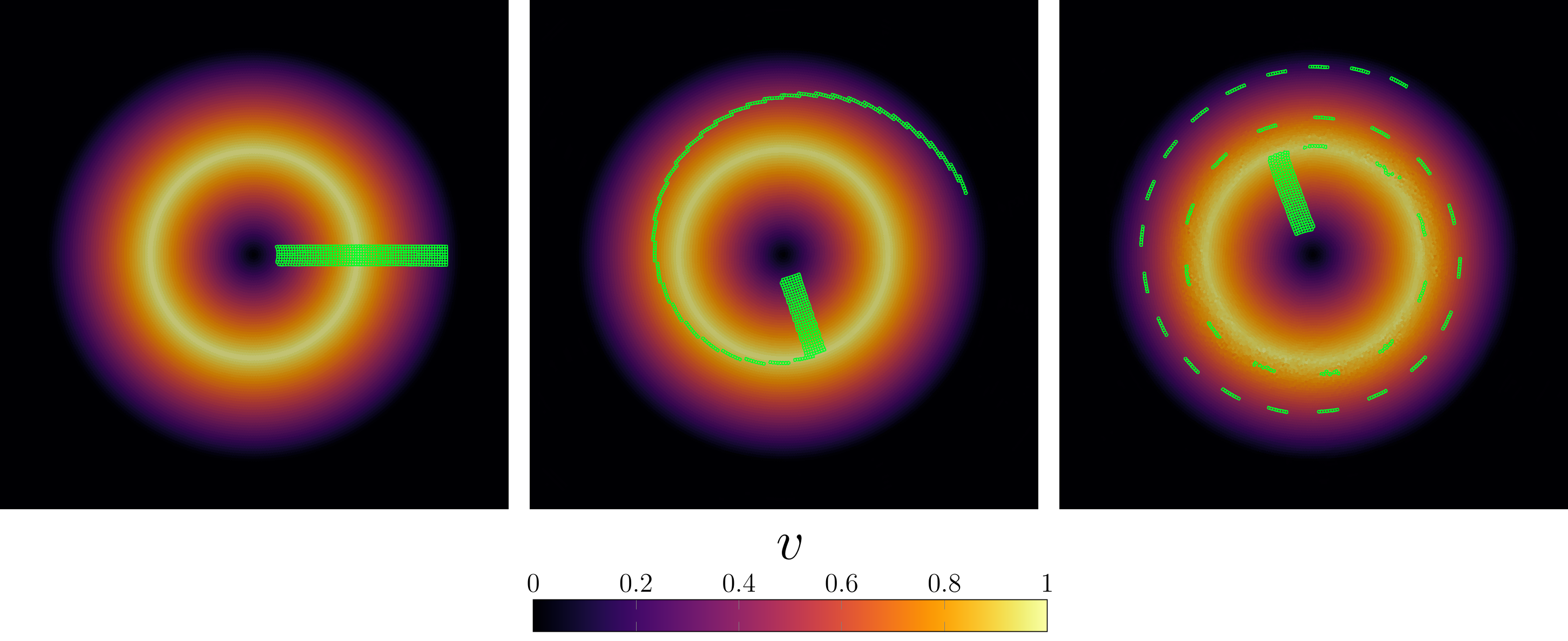}
		\caption{Gresho vortex at $t = \{0,1,3\}$. The color map indicates magnitude of velocity. A group of Voronoi cells is highlighted in green. In the last frame, the vortex structure is starting to deteriorate. The resolution in this figure is $200\times 200$.}
		\label{fig:Gresho_pics}
	\end{figure}

	\begin{figure}[ht!]
		\centering
		\begin{subfigure}{0.33\linewidth}
			\centering\includegraphics[width=0.9\linewidth]{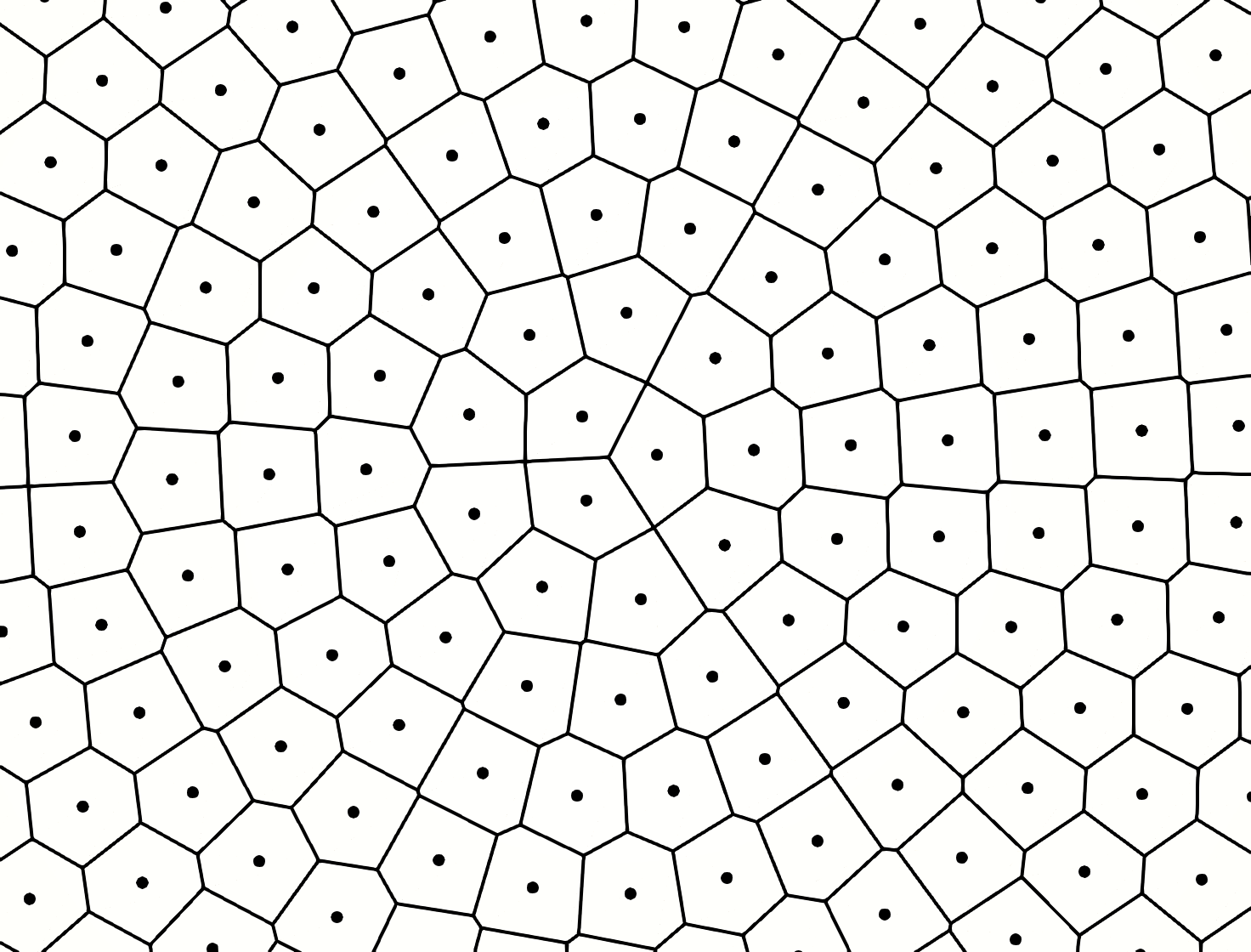}
			\caption{Initial state.}
		\end{subfigure}%
		\begin{subfigure}{0.33\linewidth}
			\centering\includegraphics[width=0.9\linewidth]{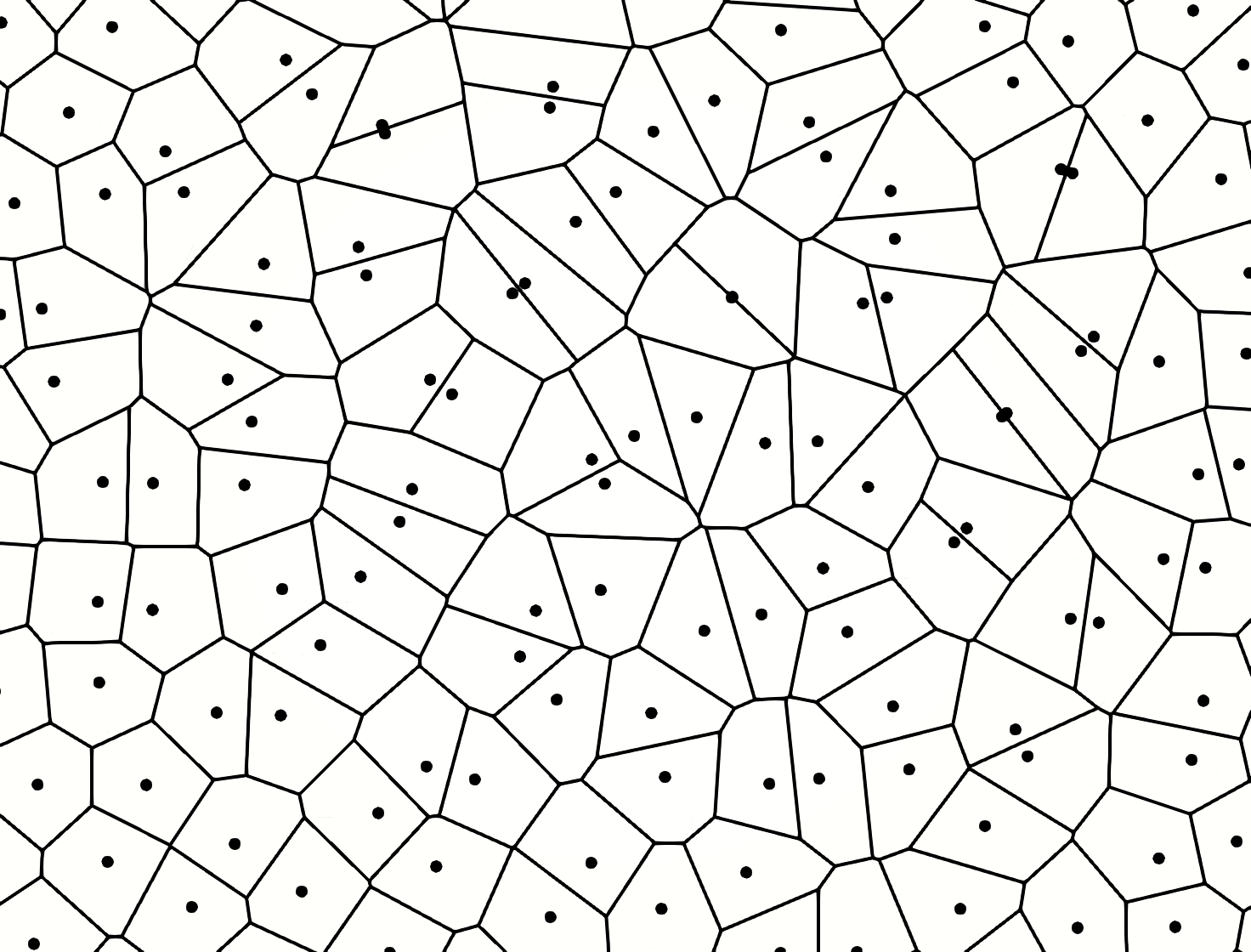}
			\caption{Unstable ($t=0.8$).}
		\end{subfigure}%
		\begin{subfigure}{0.33\linewidth}
			\centering\includegraphics[width=0.9\linewidth]{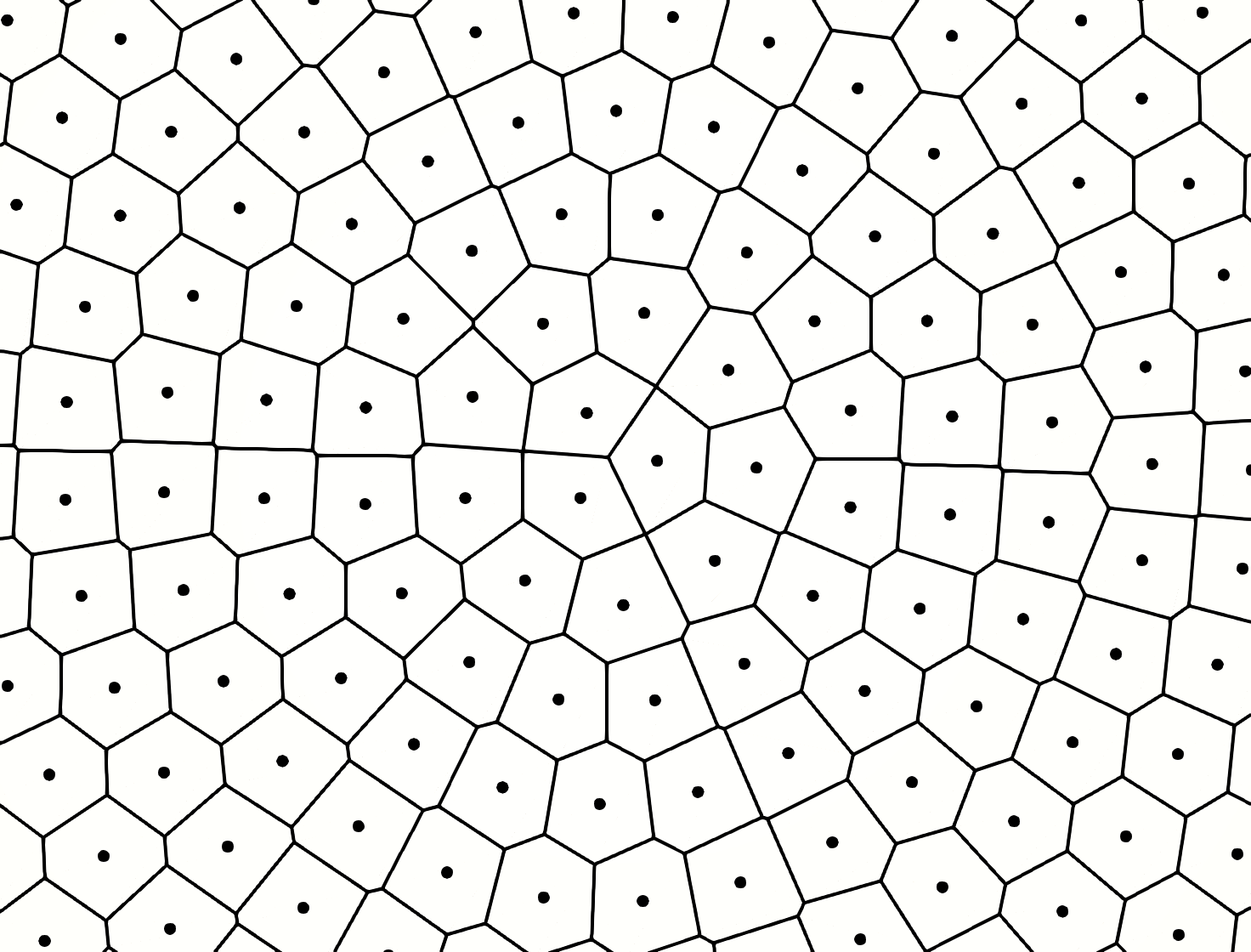}
			\caption{Stabilized ($t=0.8$).}
		\end{subfigure}
		\caption{Zoom at the vortex core at $t = 0$ and $t=0.8$ with and without stabilization. The Voronoi mesh becomes distorted as seeds are propagated away from their respective centroids.}
		\label{fig:Gresho_instability}
	\end{figure}
	
	\subsection{Lid-driven cavity}
	
	To test the correct simulation of the viscous forces, we consider the lid-driven cavity problem for an incompressible fluid. The computational domain is given by $\Omega = [-0.5, 0.5]^2$, and the fluid is initially at rest, i.e., $p = 0$ and $\vv = \mathbf{0}$. No-slip wall boundary conditions are defined on the vertical sides ($x=\pm0.5$) and at the bottom ($y=-0.5$) of the domain, while on the top side ($y=0.5$) the velocity field $\vv = (1, 0)$ is imposed. The simulations are performed as time-dependent problems on a Voronoi grid of $10000$ sample points, and four different Reynolds numbers are consider, namely $\mathrm{Re}=\{100,400,1000,3200\}$. An excellent agreement with the notorious referential solution by Ghia et al. \cite{ghia1982high} is achieved and plot in Figure \ref{fig:ldc}. More precisely, the agreement for Reynolds numbers $\mathrm{Re} = 100, 400, 1000$ is almost perfect. At $\mathrm{Re} = 3200$, we observe some discrepancy, presumably because of the numerical diffusion and insufficient resolution near boundary layers. However, SILVA is sufficiently precise to indicate subtle features of the flow, such as the left and right bottom corner vortices as well as the emergence of top left corner vortex for $\mathrm{Re} = 3200$, see Figure \ref{fig:corner}. To emphasize the Lagrangian character of SILVA, Figure \ref{fig:re1000cells} shows the trajectories of some Voronoi cells.
	
	\begin{figure}[ht!]
		\centering

		\begin{subfigure}{0.5\linewidth}
			\includegraphics[width=\linewidth]{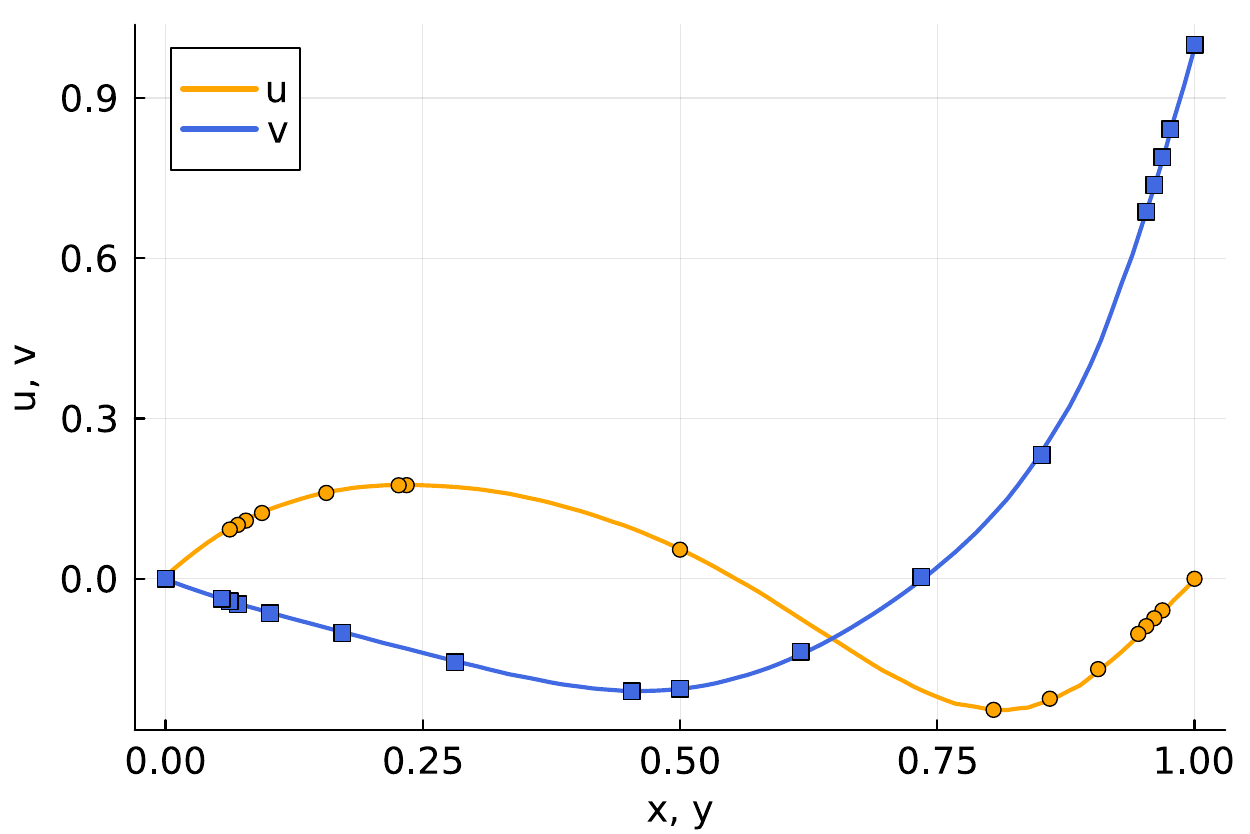} 
			\subcaption{$\mathrm{Re} = 100,\; t = 10$}
		\end{subfigure}%
		\begin{subfigure}{0.5\linewidth}
			\includegraphics[width=\linewidth]{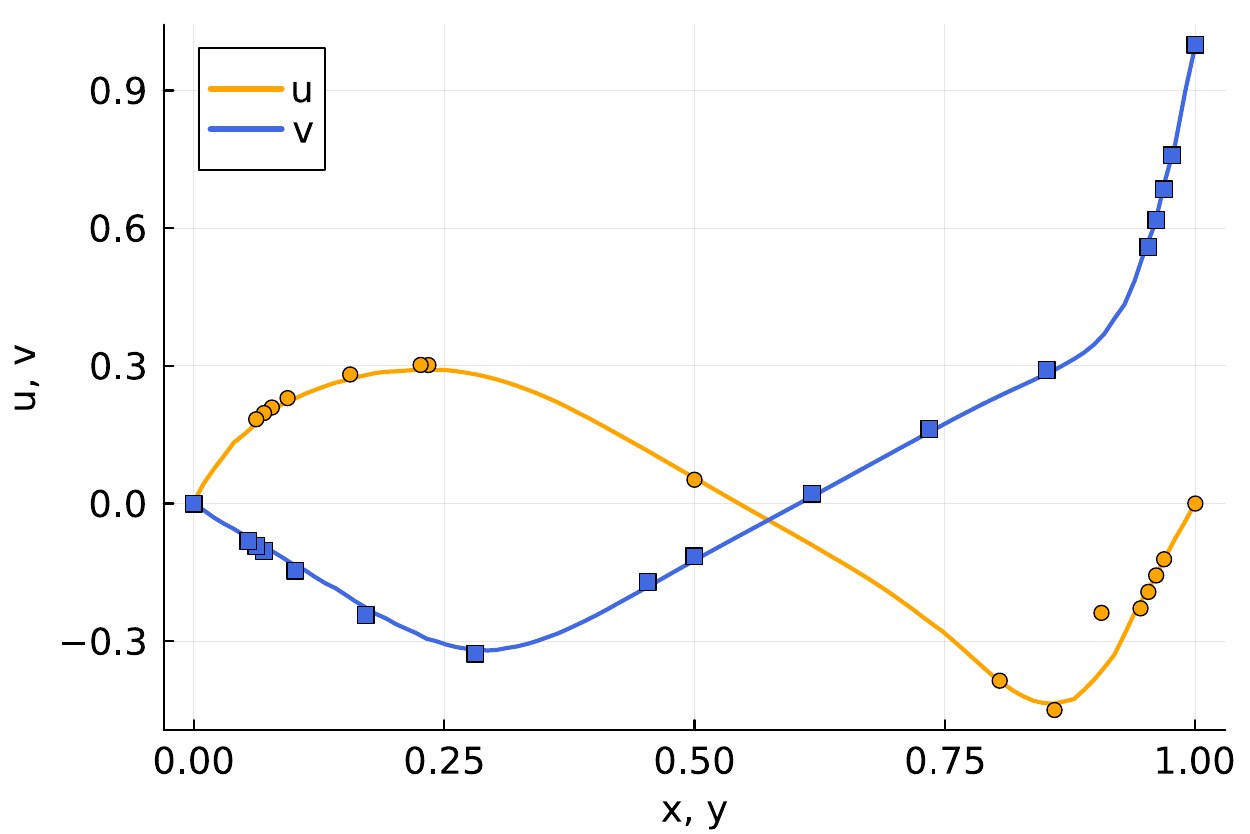} 
			\subcaption{$\mathrm{Re} = 400,\; t = 40$}
		\end{subfigure}
		\begin{subfigure}{0.5\linewidth}
			\includegraphics[width=\linewidth]{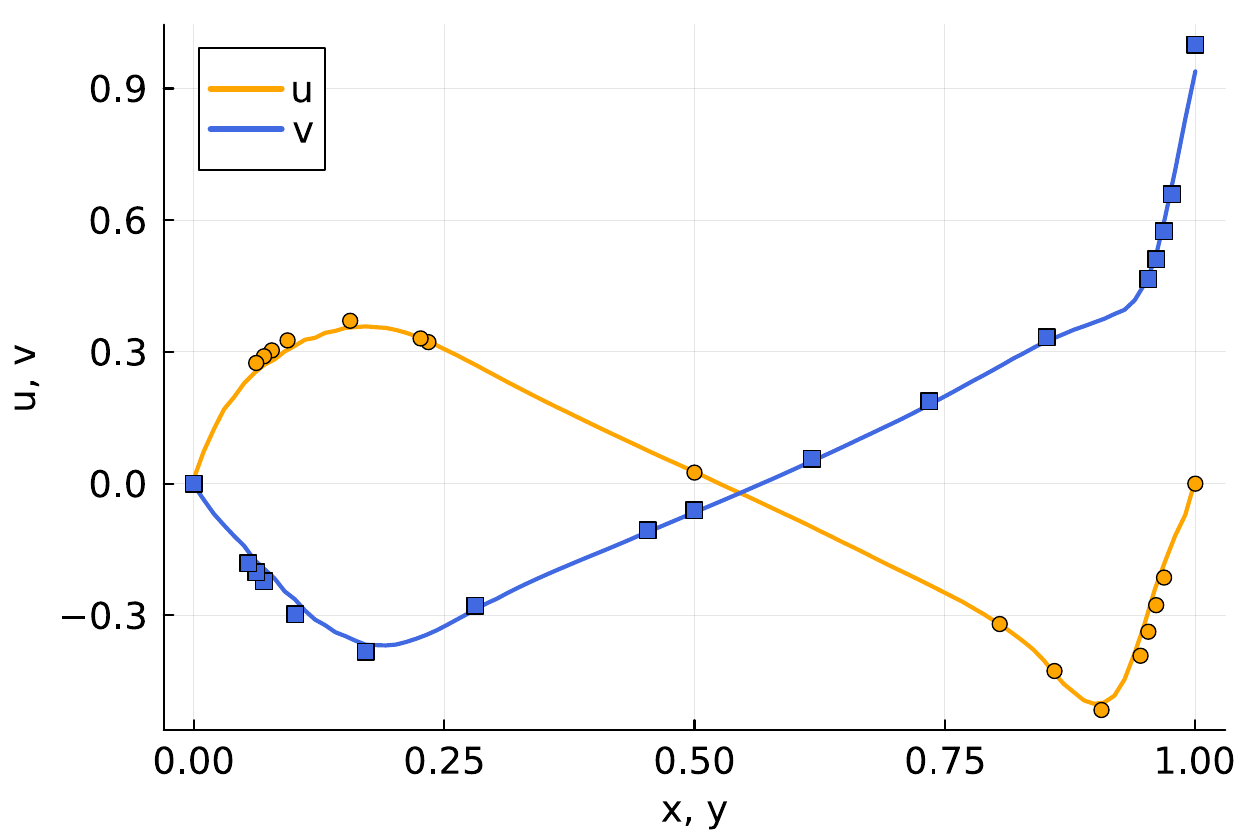} 
			\subcaption{$\mathrm{Re} = 1000,\; t = 100$}
		\end{subfigure}%
		\begin{subfigure}{0.5\linewidth}
			\includegraphics[width=\linewidth]{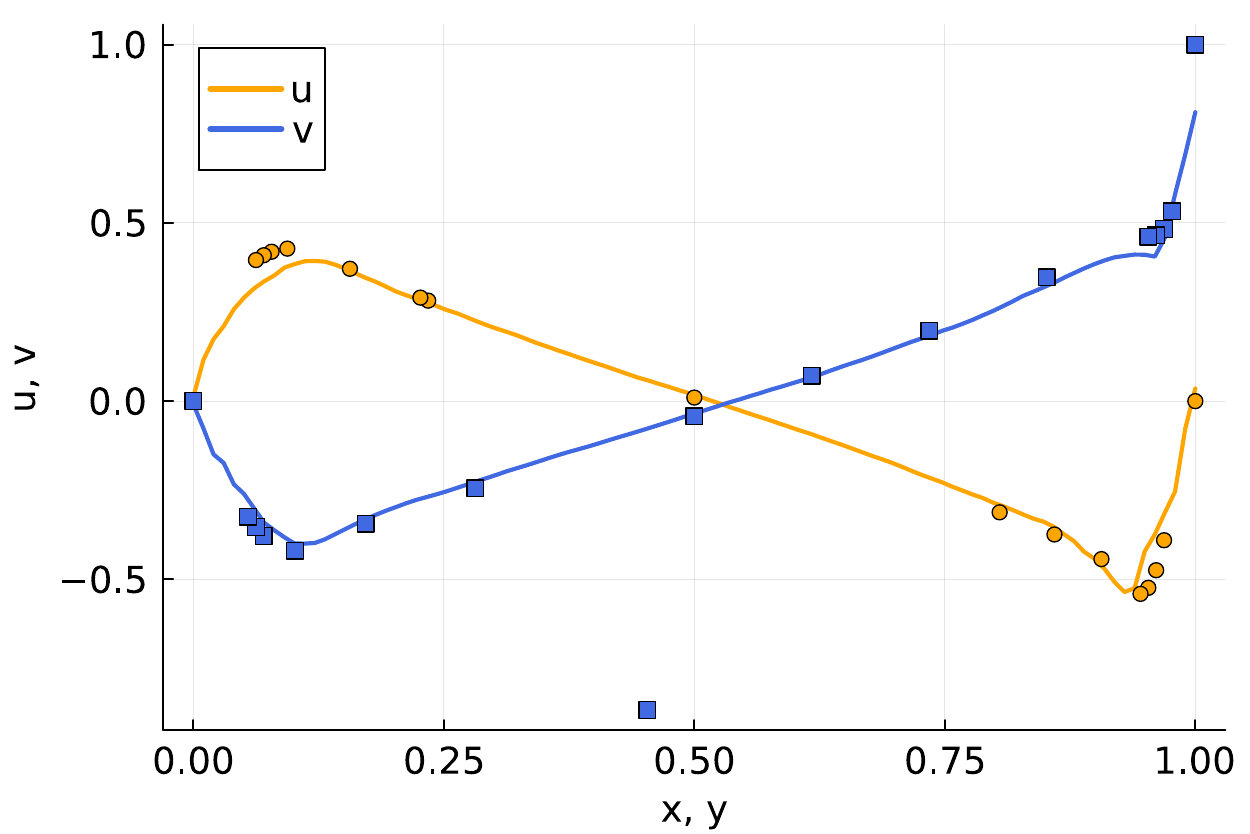} 
			\subcaption{$\mathrm{Re} = 3200,\; t = 320$}
		\end{subfigure}
		\caption{Horizontal velocity $u$ along the vertical center-line (orange) and vertical velocity $v$ along the horizontal center-line (blue) in the lid-driven cavity benchmark computed using SILVA for various Reynolds numbers. We compare the result with the referential solution of Ghia et al. \cite{ghia1982high}.}
		\label{fig:ldc}
	\end{figure}
	
	\begin{figure}[ht!]
		\centering
		\begin{subfigure}{0.5\linewidth}
			\includegraphics[trim = {4cm 0cm 4cm 0cm}, width=\linewidth]{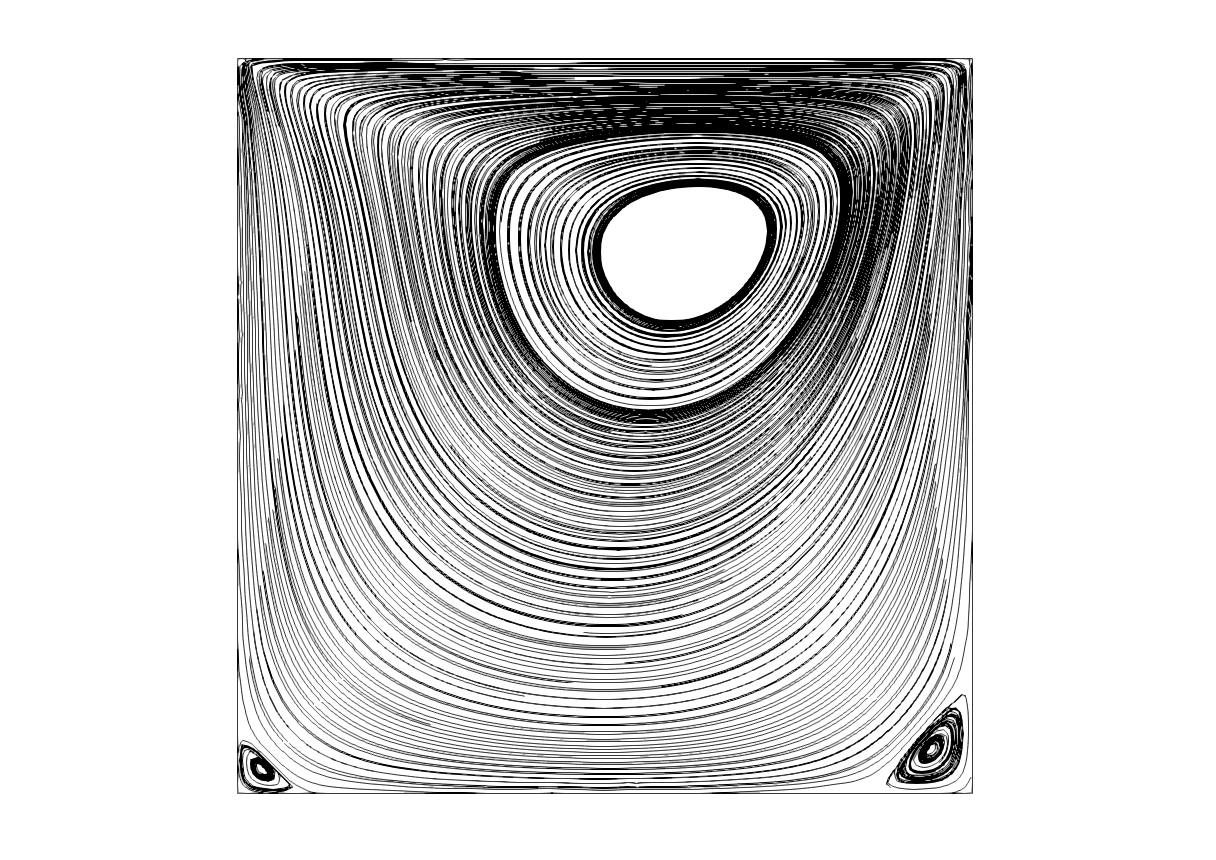} 
			\subcaption{$\mathrm{Re} = 100,\; t = 10$}
		\end{subfigure}%
		\begin{subfigure}{0.5\linewidth}
			\includegraphics[trim = {4cm 0cm 4cm 0cm}, width=\linewidth]{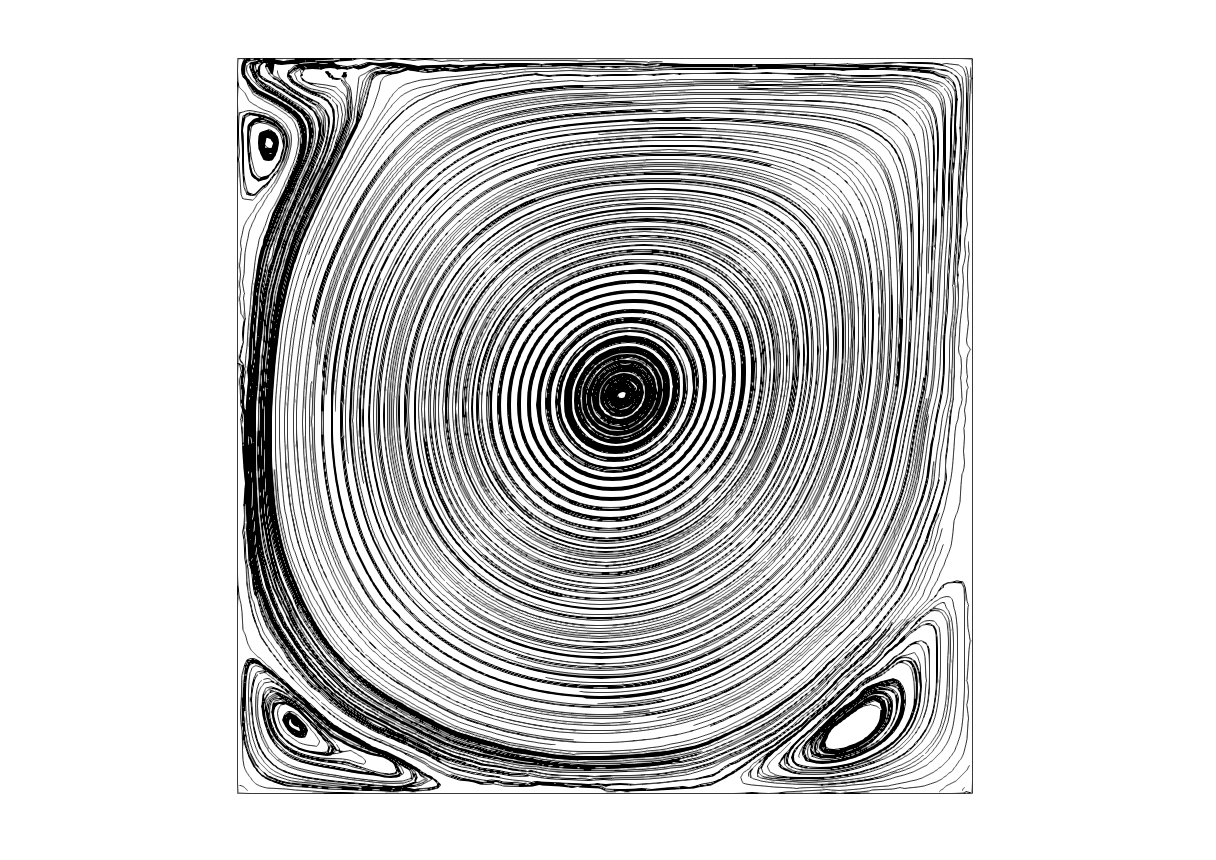} 
			\subcaption{$\mathrm{Re} = 3200,\; t = 320$}
		\end{subfigure}
		
		\caption{Streamline plots in the lid-driven cavity benchmark.}
		\label{fig:corner}
	\end{figure}

	\begin{figure}[ht!]
		\centering
		\begin{subfigure}{0.3\linewidth}
			\includegraphics[ width=\linewidth]{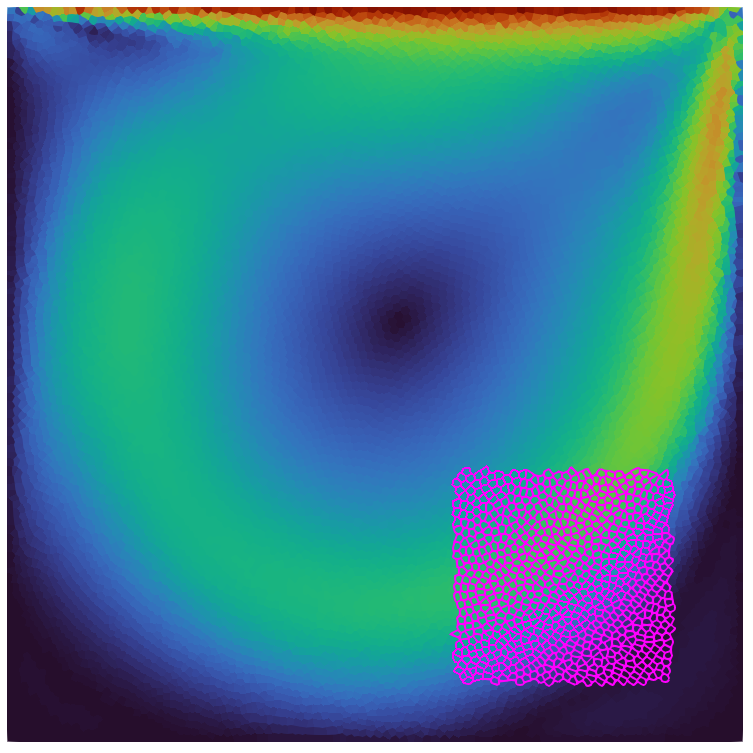} 
			\subcaption{$t = 90$}
		\end{subfigure}%
		\begin{subfigure}{0.3\linewidth}
			\includegraphics[ width=\linewidth]{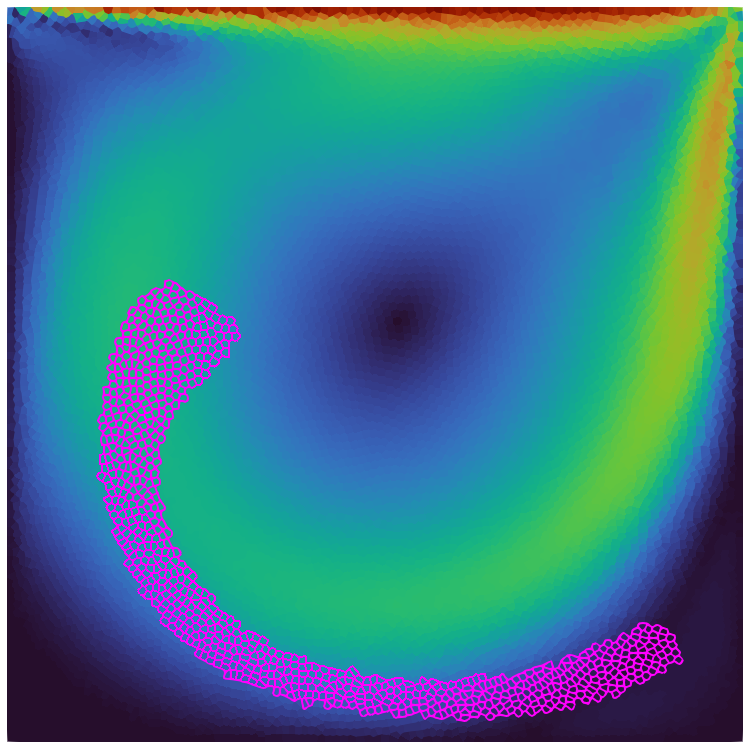} 
			\subcaption{$t = 92$}
		\end{subfigure}%
		\begin{subfigure}{0.3\linewidth}
			\includegraphics[width=\linewidth]{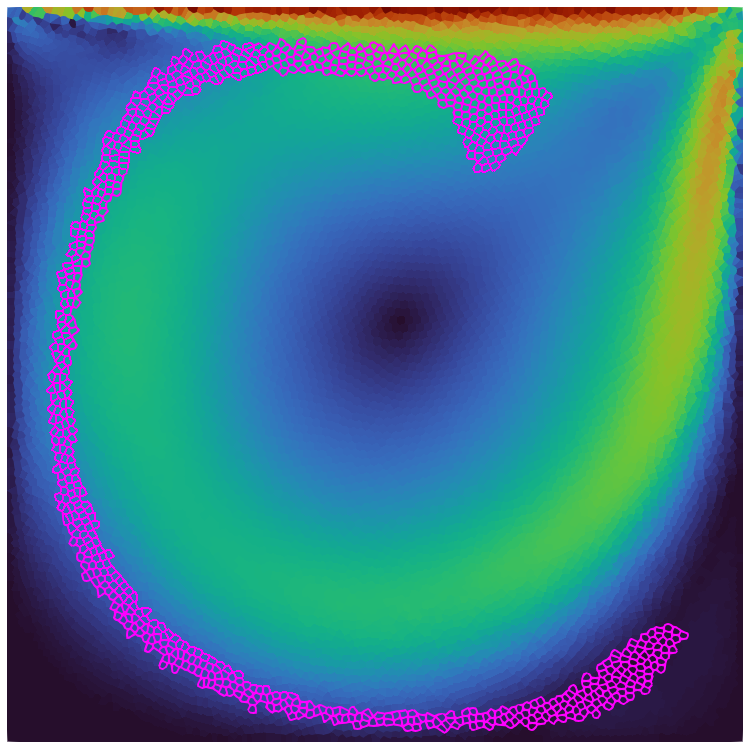} 
			\subcaption{$t = 94$}
		\end{subfigure}
		\begin{subfigure}{0.3\linewidth}
			\includegraphics[width=\linewidth]{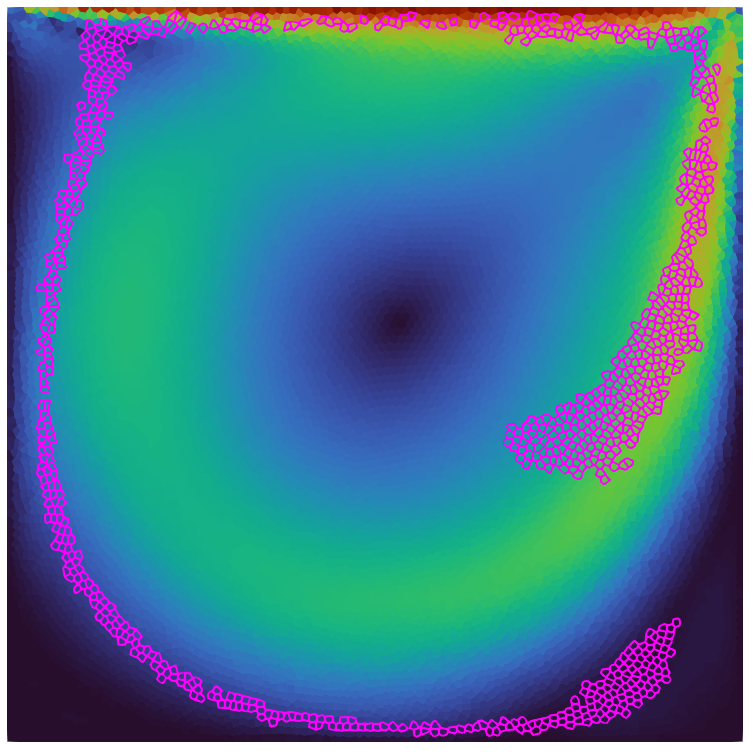} 
			\subcaption{$t = 96$}
		\end{subfigure}%
		\begin{subfigure}{0.3\linewidth}
			\includegraphics[width=\linewidth]{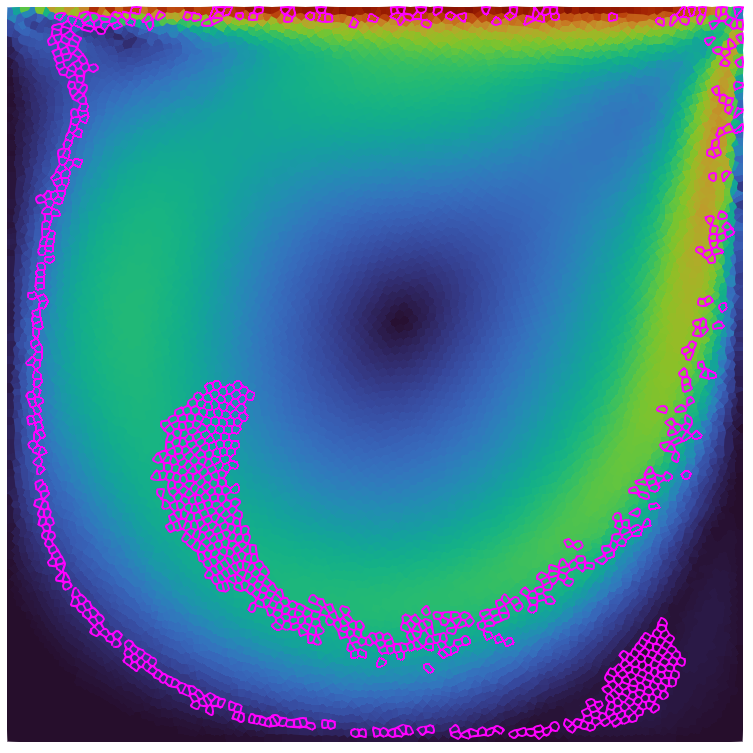} 
			\subcaption{$t = 98$}
		\end{subfigure}%
		\begin{subfigure}{0.3\linewidth}
			\includegraphics[width=\linewidth]{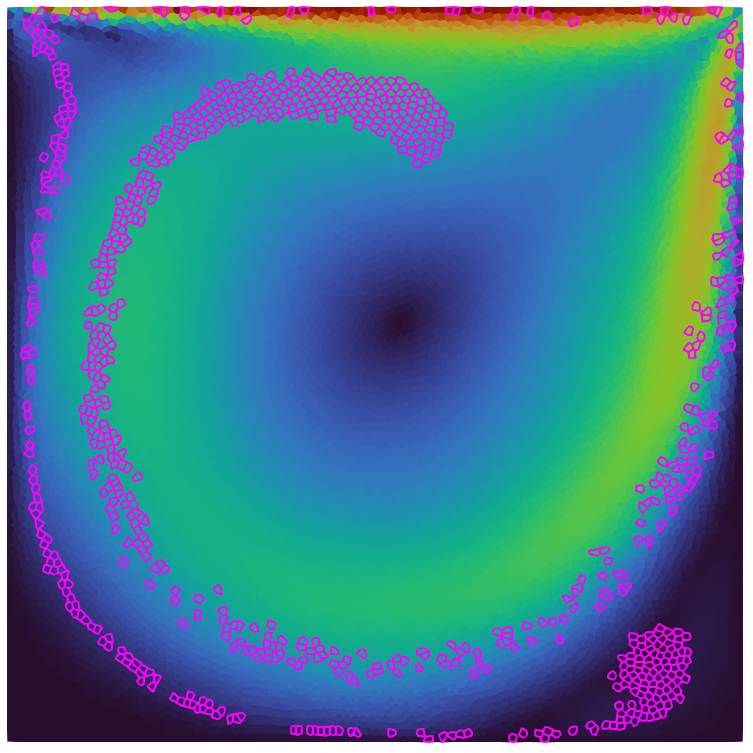} 
			\subcaption{$t = 100$}
		\end{subfigure}
		\includegraphics[width=0.4\linewidth]{images/turbo_colormap.pdf}
		\caption{Color-map of the velocity magnitude for $\mathrm{Re} = 1000$. A group of Voronoi cells is highlighted in magenta. Some of these cells remain trapped in the corner vortex.}
		\label{fig:re1000cells}
	\end{figure}

	\subsection{Rayleigh-Taylor instability}
	Finally, we present the simulation of a multi-phase problem with SILVA, namely the Rayleigh-Taylor instability \cite{rayleigh1882investigation} of an inviscid fluid without surface tension. Our setup is the same as in the SPH study presented in \cite{hu2007incompressible}. The domain is the rectangle $\Omega = [0,1]\times [0,2]$ which is discretized with $80000$ Voronoi seeds. The initial density field is given by
	\begin{equation}
		\rho(x,y) = \begin{cases}
			1.8 & y > \phi(x)\\
			1 & y < \phi(x)
		\end{cases}.
	\end{equation}
	Two immiscible fluids of different densities are involved, which are separated by a curve described as
	\begin{equation}
		\phi(x) = 1 - 0.15 \cos (2\pi x). 
	\end{equation}
	The densities correspond to Atwood number $\mathrm{A} = 2/7$. The two fluids begin at rest, have identical viscosity and are subjected to a homogeneous gravitational field. The Reynolds number is $\mathrm{Re} = 420$ and the Froude number is $\mathrm{Fr} = 1$. No-slip boundary conditions are imposed on all sides of the domain.
	
	This problem is concerned with an incompressible fluid but with a \textit{heterogeneous} density field. Hence, the approximation \eqref{eq:approxapprox} is no longer valid when we need to take into account the gradient of pressure. The correct discrete equation is now given by 
	\begin{equation}
		-\lap{p}{i}^{n+1} = -\frac{\rho_i}{\Delta t} \wdiv{\vv}{i}^n - \frac{1}{\rho_i} \sgrad{\rho}{i} \cdot \sgrad{p}{i}^{n+1}.
		\label{eq:equation_for_p}
	\end{equation}
	The above system can be solved quite efficiently by means of fixed-point iterations, with $m$ denoting the current iteration number:
	\begin{equation}
		-\lap{p^{n+1,(m+1)}}{i} = -\frac{\rho_i}{\Delta t} \wdiv{\vv}{i}^n - \frac{1}{\rho_i} \sgrad{\rho}{i} \cdot \sgrad{p^{n+1,(m)}}{i},
		\label{eq:fixpi}
	\end{equation}
	starting with $p^{n+1,(0)} = p^n$. The iterative process stops when the condition $|p^{n+1,(m+1)}-p^{n+1,(m)}|<10^{-12}$ is satisfied. Figure \ref{fig:rti} displays the result of SILVA computation, which are in qualitative agreement with the ISPH solution \cite{hu2007incompressible}. We also perform a mesh convergence study of this test for three different mesh resolutions ($N=\{60,100,200\}$), and the results are depicted in Figure \ref{fig:rti_refine}. We observe that the interface profile between the two immiscible fluids becomes increasingly more smooth, and the numerical artifacts disappear.
	
	\begin{figure}[ht!]
		\centering
		\begin{subfigure}{0.33\linewidth}
			\centering
			\includegraphics[width=0.8\linewidth]{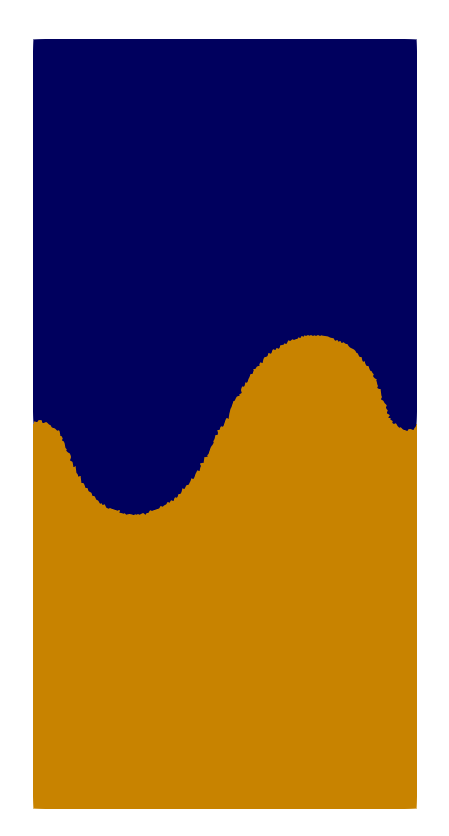}
			\subcaption{$t=1$}
		\end{subfigure}%
		\begin{subfigure}{0.33\linewidth}
			\centering
			\includegraphics[width=0.8\linewidth]{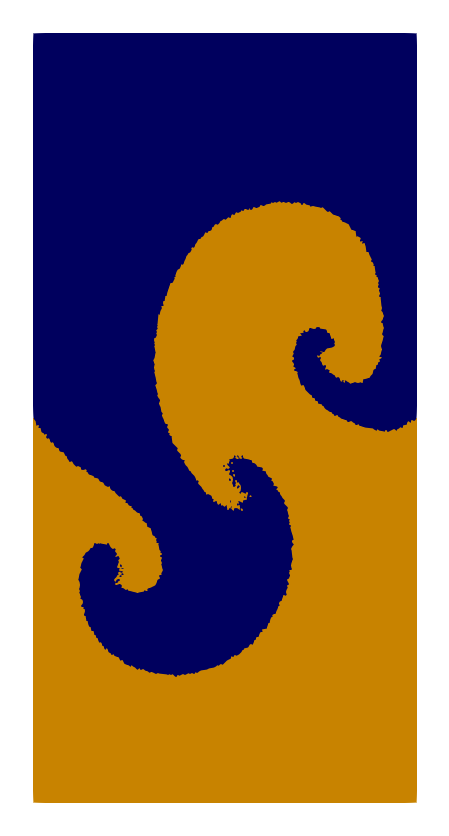}
			\subcaption{$t=3$}
		\end{subfigure}%
		\begin{subfigure}{0.33\linewidth}
			\centering
			\includegraphics[width=0.8\linewidth]{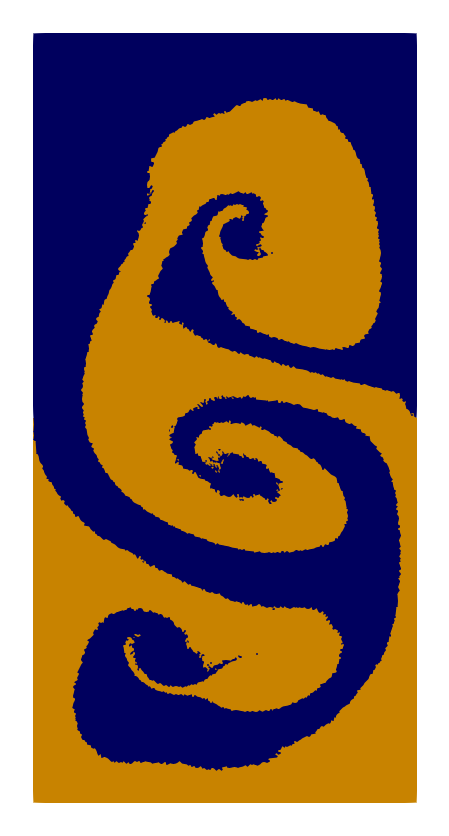}
			\subcaption{$t=5$}
		\end{subfigure}
		\caption{Snapshots of the Rayleigh-Taylor instability at output times $t=\{1,3,5\}$. The lighter fluid is orange and the heavier is dark blue. The simulation contained $80000$ Voronoi cells.}
		\label{fig:rti}
	\end{figure}
	
	\begin{figure}[ht!]
		\centering
		\begin{subfigure}{0.33\linewidth}
			\centering
			\includegraphics[width=0.8\linewidth]{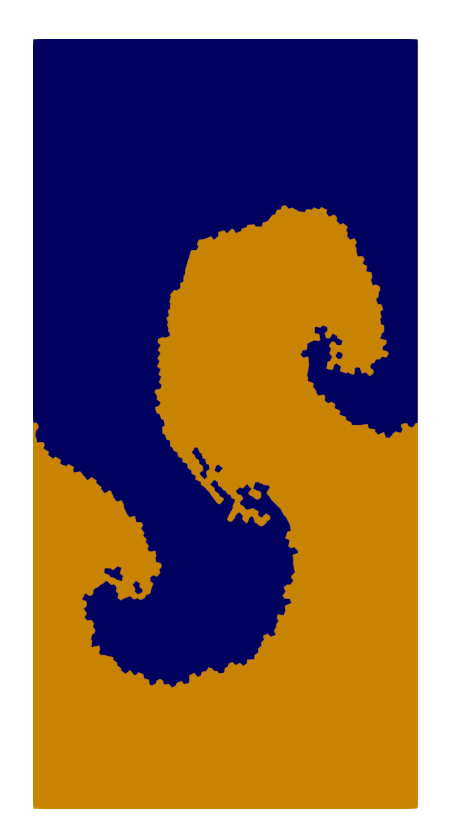}
			\subcaption{$N = 60$}
		\end{subfigure}%
		\begin{subfigure}{0.33\linewidth}
			\centering
			\includegraphics[width=0.8\linewidth]{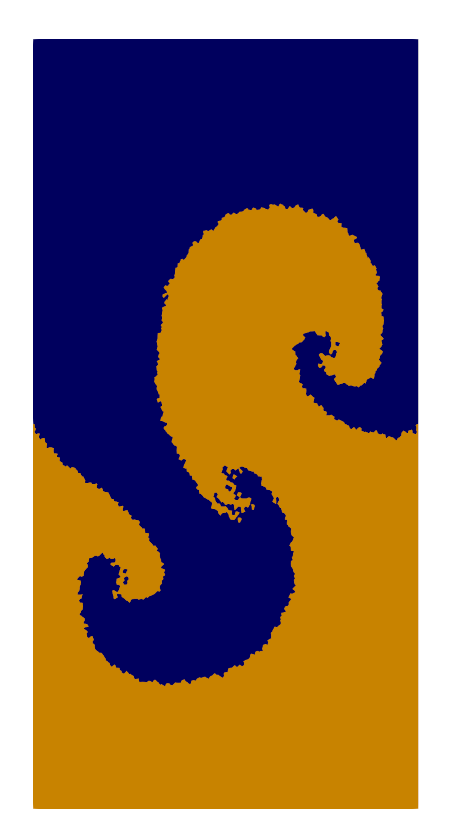}
			\subcaption{$N = 100$}
		\end{subfigure}%
		\begin{subfigure}{0.33\linewidth}
			\centering
			\includegraphics[width=0.8\linewidth]{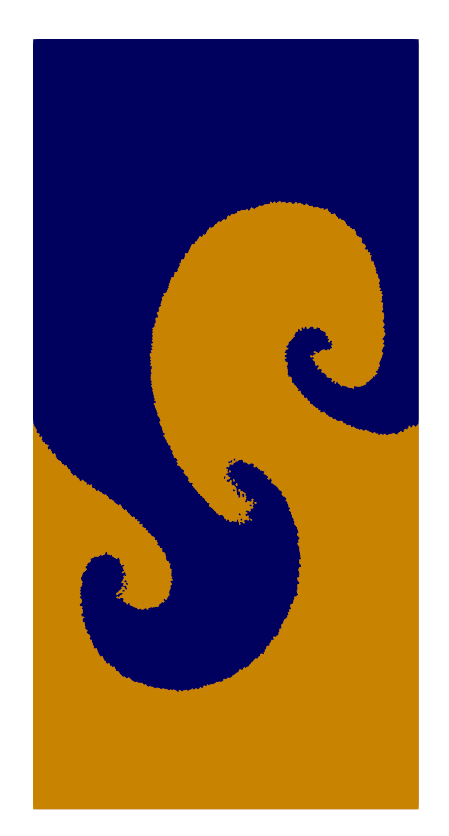}
			\subcaption{$N = 200$}
		\end{subfigure}
		\caption{The Rayleigh-Taylor instability at $t=3$ for three different resolutions. $N$ is the number of Voronoi seeds per shorter side of the domain, so there are $2N^2$ Voronoi seeds in total.}
		\label{fig:rti_refine}
	\end{figure}

	\section{Conclusion} \label{sec.concl}
	In this paper, a new Lagrangian Voronoi method (LVM) for the solution of the incompressible Navier-Stokes equations has been presented. The spatial discretization is carried out by means of moving Voronoi cells that are regenerated at each time step with an efficient cell-list based Voronoi mesh generation technique, optimized for multi-threaded time-dependent problems. To overcome the problem related to the vortex-core instability associated with the established LVM gradient operator, we have proposed a new stabilizing procedure, which is notable by the lack of numerical parameters and by preserving the exactness of linear gradients. A semi-implicit time discretization has been developed for the first time in the context of Lagrangian Voronoi methods, which discretizes explicitly the trajectory equation, that drives the mesh motion, and the viscous terms in the velocity equation. To enforce the projection of the velocity onto the associated divergence-free manifold, the velocity equation is inserted in the divergence-free constraint, yielding an implicit linear system for the pressure that corresponds to the well-known elliptic Poisson equation. The matrix of the system is provably symmetric and can therefore be solved with very efficient iterative linear solvers like the Conjugate Gradient method or MINRES. The accuracy of the scheme has been assessed in four different test-cases, featuring both viscous and inviscid fluids as well as multi-phase flow. Linear convergence has been empirically studied in a Gresho vortex benchmark. Our novel scheme, referred to as SILVA, is able to run long time simulations without any mesh degeneration, inspired by the so-called reconnection based algorithms on moving meshes.
	
	SILVA is an interesting alternative to the incompressible SPH method, because of the benefits coming from a sparser matrix associated to the elliptic Poisson system, and a relatively easy implementation of free-slip and Dirichlet boundary conditions. Multi-phase problems with sharp interface reconstruction and fluid-structure interaction are among the most promising applications. So far, a limitation of our scheme is the requirement of domain convexity and the inability to treat free surfaces. The benchmarks in this paper are two-dimensional, although we do not foresee major obstacles regarding the extension to three-dimensional problems. Further research could be directed towards improving the order of convergence, both in space (by means of a second-order reconstruction) and in time (using high order IMEX schemes, such as CNAB integrator or implicit-explicit midpoint rule \cite{carlino2024arbitrary}). Compressible flows will also be addressed in the future, by solving a mildly non-linear system for the pressure. For the sake of simplicity, we have used an explicit treatment of the viscous terms, but an implicit approach would be preferable for flows with very small Reynolds number, as presented in \cite{SICNS}.
	
	\section*{Acknowledgments}
	This work was financially supported by the Italian Ministry of University 
	and Research (MUR) in the framework of the PRIN 2022 project No. 2022N9BM3N. WB received financial support by Fondazione Cariplo and Fondazione CDP (Italy) under the project No. 2022-1895. IP and WB are members of the INdAM GNCS group in Italy.

\end{document}